\documentclass[amscd,amssymb,verbatim,12pt]{amsart}
\numberwithin{equation}{section}
\usepackage{color}
\usepackage{pst-math}
\usepackage{pst-func}
\usepackage{pst-3dplot}
\usepackage{amsmath}
\usepackage{a4,amssymb}
\usepackage[makeroom]{cancel}
\usepackage[normalem]{ulem}

\newtheorem{thm}{Theorem}[section]
\newtheorem{cor}[thm]{Corollary}
\newtheorem{lem}[thm]{Lemma}

\theoremstyle{definition}

\newtheorem{rem}[thm]{Remark}

\newif\ifShowLabels
\ShowLabelstrue
\newdimen\theight
\def\TeXref#1{
     \leavevmode\vadjust{\setbox0=\hbox{{\tt
            \quad\quad  {\small  \bf #1}}}%
     \theight=\ht0
     \advance\theight  by  \dp0
     \advance\theight  by  \lineskip
     \kern -\theight \vbox  to
     \theight{\rightline{\rlap{\box0}}%
      \vss}%
      }}%

    {\begin{thm}\label{#1} \ifShowLabels \TeXref{#1} \fi}%
    {\end{thm}}

    {\begin{def}\label{#1} \ifShowLabels \TeXref{#1} \fi}%
    {\end{def}}

    {\begin{lem}\label{#1} \ifShowLabels \TeXref{#1} \fi}%
    {\end{lem}}

    {\begin{cor}\label{#1} \ifShowLabels \TeXref{#1} \fi}%
    {\end{cor}}

\newcommand{\eqRef}[1]%
     {\ifShowLabels \TeXref{#1} \fi
      \begin{equation}\label{#1} }

\ShowLabelsfalse

\newcommand{\vsp}{\vskip 1em}
\newcommand{\vspp}{\vskip 2em}
\newcommand{\NI}{\noindent}
\newcommand{\bea}{\begin{eqnarray}}
\newcommand{\eea}{\end{eqnarray}}
\newcommand{\IR}{\mathbb{R}}
\newcommand{\bas}{\begin{align*}}
\newcommand{\eas}{\end{align*}}
\newcommand{\ba}{\begin{align}}
\newcommand{\ea}{\end{align}}

\newcommand{\be}{\begin{equation}}
\newcommand{\ee}{\end{equation}}
\newcommand{\ben}{\begin{eqnarray*}}
\newcommand{\een}{\end{eqnarray*}}
\newcommand{\lam}{\lambda}
\newcommand{\Om}{\Omega}
\newcommand{\om}{\omega}
\newcommand{\tht}{\theta}

\newcommand{\al}{\alpha}
\newcommand{\bt}{\beta}
\newcommand{\bb}{\bar{\beta}}

\newcommand{\Lam}{\Lambda}

\newcommand{\Lmn}{\Lambda_{\min}}

\newcommand{\Lmx}{\Lambda_{\max}}

\newcommand{\Ls}{\Lambda^{\sup}}

\newcommand{\Ln}{\Lambda^{\inf}}

\newcommand{\g}{\gamma}
\newcommand{\gs}{\gamma^*}

\newcommand{\ve}{\varepsilon}
\newcommand{\dl}{\delta}

\newcommand{\s}{\sigma}

\newcommand{\Pc}{\mathcal{P}_\sigma}

\newcommand{\Pz}{\mathcal{P}_0}

\newcommand{\st}{\sigma^*}

\title[Phragm\'en-Lindel\"of]{A Phragm\'en-Lindel\"of property of viscosity solutions to a class of doubly nonlinear parabolic equations: Bounded Case}

\author[Bhattacharya and Marazzi]{Tilak Bhattacharya and Leonardo Marazzi}

\begin{document}

\maketitle

\begin{abstract} We study Phragm\'en-Lindel\"of properties for viscosity solutions to a class of nonlinear 
parabolic equations of the type $H(Du, D^2u+Z(u)Du\otimes Du)+\chi(t)|Du|^\s-u_t=0$ under a certain boundedness condition on $H$. We also state results for positive solutions to a class of doubly nonlinear equation $H(Du, D^2u)-f(u)u_t=0$.
\end{abstract}

\section{\bf Introduction}

In this work we study the Phragm\'en-Lindel\"of property of viscosity solutions $u(x,t)$ for a class of nonlinear parabolic equations on the infinite strip $\IR^n_T=\IR^n\times (0,T)$, where $n\ge 2$ and $0<T<\infty$. The current work may be viewed as partly complementing the work \cite{BM6}. See also, \cite{BM3}.

Set $\IR^n_T=\IR^n\times (0,T)$ and let $g:\IR^n\rightarrow (0,\infty)$ be continuous and $f:[0,\infty)\rightarrow [0,\infty)$ be an increasing continuous function. As described in \cite{BM6}, the motivation for this work arises from the study of doubly nonlinear equations of the kind
\bea\label{sec1.1}
H(Du, D^2u)-f(u)u_t=0, \;\;\mbox{in $\IR^n_T,$}\;\;\mbox{with $u(x,0)=g(x),\;\forall x$ in $\IR^n$,}
\eea
where $H$ satisfies certain homogeneity conditions and $u\in C(\IR^n\times [0,T))$ is a viscosity solution. See Section 2 for more details. 

As noted in \cite{BM5, BM6}, if $f$ satisfies certain conditions then there is an increasing function $\phi$ and a non-increasing function $Z\ge 0$ such that the change of variable $u=\phi(v)$ transforms the differential equation in (\ref{sec1.1}) to 
\eqRef{sec1.2}
H(Dv, D^2v+Z(v)Dv\otimes Dv)-v_t=0, \;\mbox{in $\IR^n_T$ with $v(x,0)=\phi^{-1}(g(x)),\;\forall x$ in $\IR^n$.}
\ee
It follows that the solutions of (\ref{sec1.2}) and hence, the solutions of (\ref{sec1.1}), satisfy a comparison principle, see \cite{BM1, BM2, BM5}. Incidentally, we do not require that $Z$ be defined in all of $\IR$, a matter that will be discussed later. For purposes of the current discussion, we will overlook this issue.

As done in \cite{BM6}, we consider a some what more general setting and study Phragm\'en-Lindel\"of type results for equations of the kind
\bea\label{sec1.3}
&&H(Dv, D^2v+Z(v)Dv\otimes Dv)+\chi(t) |Dv|^{\s}-v_t=0, \;\;\mbox{in $\IR^n_T$},\nonumber\\
&&\qquad\mbox{$v(x,0)=h(x)$}, \;\forall x\in \IR^n,
\eea
where $\s\ge 0$ and $\chi:(0,T)\rightarrow \IR$ and $h:\IR^n\rightarrow \IR$ are both continuous and bounded.

In \cite{BM6}, we assumed that $\sup_\lam \left[ \min_{|e|=1}H(e, \lam e\otimes e+I) \right]=\infty$, where $e$ is a unit vector, $I$ is the $n\times n$ identity matrix and $\lam$ is a real valued parameter. We showed that the maximum principle was valid for solutions that satisfied certain growth rates for large $x$. The class of operators, we considered, included, among others, the $p$-Laplacian ($p\ge 2$), the infinity-Laplacian and the Pucci operators. 
The current work addresses the case $\sup_\lam \left[ \max_{\{e=1\}}|H(e, \lam e\otimes e\pm I)| \right]<\infty$ and, in a sense, complements \cite{BM6}. In Section 2, we have listed some examples of operators that satisfy this condition. 

We remark also that, much like \cite{BM6}, the imposed growth rates are influenced by the dueling terms $Z(v)Dv\otimes Dv)$ and $\chi(t)|Dv|^\s$ and the power $\s$.
Since $Z\ge 0$, by ellipticity, $H(Du, D^2u)\le H(Du, D^2u+Z(u)Du\otimes Du)$. Our work will show that, unlike \cite{BM6}, $Z(s)$ can be allowed to vanish, i.e, 
$Z(s)=0,\;\forall s\ge s_0$, for some $s_0$. The value of $Z$ does not influence the bound on $H(e, \lam e\otimes e\pm I)$.

We have divided our work as follows. In Section 2, we introduce more notation and state the main results. Section 3 contains preliminary calculations and previously proven lemmas, useful for the current work, In Sections 4 and 5, we present the constructions of super-solutions and sub-solutions respectively. Section 6 addresses some special situations. The proofs of the main results appear in Section 7.   

As a final note, we do not address questions of existence and uniqueness and nor do
we address optimality of the growth rates stated in the theorems. Also, we direct the
reader to \cite{AJK, ED, JL, TR, Tych} for related questions and discussion.

\vspp
\section {Notation and main results}
\vsp
In this work, sub-solutions, super-solutions and solutions are meant in the sense of viscosity. For definitions, we direct the reader to \cite{BM5,CIL}.

We introduce notation that are used throughout this work. We address the problems in (\ref{sec1.1}) and (\ref{sec1.3}) on infinite strips in $\IR^{n+1}$ where $n\ge 2$. The letter $o$ denotes the origin in $\IR^n$ and $e$ denotes a unit vector in $\IR^n$. Let $S^{n\times n}$ be the set of all symmetric $n\times n$ real matrices. Let $I$ be the identity matrix and $O$ the $n\times n$ zero matrix. The expressions $usc$ and $lsc$ stand for {\it upper semi-continuous} and {\it lower semi-continuous} respectively.  

Through out this work, we assume that $H$ satisfies the following conditions. 
\vsp
{\bf Condition A (Monotonicity):} Let $H:\IR^n\times S^{n\times n}\rightarrow \IR$ is continuous for any $(q,X)\in \IR^n\times S^{n\times n}$. We require that
\bea\label{sec2.1}
&&\mbox{(i)}\;H(q,X)\le H(q,Y),\;\mbox{$\forall\;q\in \IR^n$ and $\forall\;X,\;Y$ in $S^{n\times n}$, with $X\le Y$}, \nonumber\\
&&\mbox{(ii)}\;H(q,O)=0,\;\mbox{$\forall\;q\in \IR^n$.}
\eea
Clearly, for any $q\in \IR^n$ and $X\in S^{n\times n}$, $H(q,X)\ge 0$ if $X\ge O$.
\vsp
{\bf Condition B (Homogeneity):} There is a constant $k_1\ge 0,$ such that for any $(q,X)\in \IR^n\times S^{n\times n}$,
\bea\label{sec2.2}
&&\mbox{(i)}\;H(\tht q, X)=|\tht|^{k_1}H(q, X),\;\mbox{$\forall\;\tht\in \IR$, and}\nonumber\\
&&\mbox{(ii)}\;H(q, \tht X)=\tht H(q, X),\;\mbox{$\forall\;\tht>0.$}
\eea
\vsp
We introduce two quantities before stating the next condition. For any unit vector $e\in \IR^n$, we recall that $(e\otimes e)_{ij}=e_ie_j,$ for any $i,j,=1,2, \cdots, n.$ Moreover, $e\otimes e\ge O$. For $\lam\in \IR$, set 
\bea\label{sec2.3}
\Lambda_{\min}(\lam)=\min_{|e|=1}H(e, \lam e\otimes e-I)\;\;\mbox{and}\;\;\Lambda_{\max}(\lam)=\max_{|e|=1}H(e, \lam e\otimes e+ I). 
\eea
By Condition A, both $\Lmn(\lam)$ and $\Lmx(\lam)$ are non decreasing functions of $\lam$. 
\vsp
{\bf Condition C(Growth at Infinity):} We impose that
$$\max_{|e|=1}H(e, -I)<0<\min_{|e|=1}H(e, I).$$
Set $\Ls=\sup_{\lam} \Lmx(\lam)$ and $\Ln=\inf_{\lam}\Lmn(\lam)$.
Assume further that 
\eqRef{sec2.4}
 \Ls<\infty.
\ee
It follows easily from (\ref{sec2.4}), Condition A and Condition B (ii) that $H(e,  e\otimes e)=0.$
\vsp
In this work, the requirement (\ref{sec2.4}) will apply through out. For some of the results, we will require additionally that 
$$\Ln>-\infty.\qquad\qquad \Box$$
\vsp
We now present examples of operators that satisfy Conditions A, B and C, and include some observations.  

\vsp
\begin{rem}\label{sec2.40}  (i) An example of an operator that satisfies
Conditions A, B and C is 
$$
H_p(q, X)=|q|^p \{ |q|^2Tr(X)-q_iq_jX_{ij} \},\;p\ge 0,\;\;\forall (q, X)\in \IR^n\times S^{n\times n},$$
where $Tr(X)$ is the trace of $X$. Clearly,
$$H_p(Du, D^2u)=|Du|^p\left( |Du|^2\Delta u-\Delta_\infty u \right).$$
Thus, for any $c\in \IR$,
\ben
H_p(q, X+cq\otimes q)=|q|^p\left[ |q|^2Tr(X)+c|q|^4-q_iq_jX_{ij}-c|q|^4 \right]=H_p(q, X).
\een
In particular,
\ben
H_p(e, \lam e\otimes e \pm I)=H_p(e, \pm I)=\pm(n-1),\;\;\mbox{for any}\; \lam\in \IR.
\een
Note that $k=k_1+1\ge 1$, see (\ref{sec2.51}) below. A closely allied example is $H(Du, D^2u)=|Du|^4\Delta_pu-(p-1)|Du|^p\Delta_\infty u.$

(ii) A second example can be constructed as follows. Let $\mu_i=\mu_i(X),\;i=1,2,\cdots,n$ be the eigenvalues of any $X\in S^{n\times n}$. We order these as $\mu_1\ge \mu_2\ge \cdots\ge \mu_n$. Define
$$H^m_{p}(q,X)=|q|^p \left(\sum_{i=m}^n \mu_i(X) \right),\quad \mbox{$p\ge 0$ and $2\le m<n$.} $$
Clearly, $H$ satisfies Conditions A and B, $H^m_{p}(e,\pm I)=\pm(n-m+1)$. 

Observe that det$(e\otimes e)=0$ and $(e\otimes e)^2=e\otimes e$ and $(e\otimes e-\mu I)x=0\, (x\perp e)$ if and only if $\mu=0$ or $\mu=1\, (x\parallel e)$ implying that the eigenvalues of $e\otimes e$ are $0$ (multiplicity $n-1$) and $1$. Thus, 
the eigenvalues of $\lam e\otimes e+I$ are $1$ (multiplicity $n-1$) and $\lam+1$. Similarly, the eigenvalues of $\lam e\otimes e-I$ are $-1$ (multiplicity $n-1$) and $\lam-1.$ Thus,
$$\left\{ \begin{array}{ccc}H^m_{p}(e,\lam e\otimes e+I)=n+1-m,&\lam\ge 0,\\ H^m_p(e, \lam e\otimes e-I)=\lam-(n-m+1), & \lam\le 0. \end{array}\right. $$
Some of our results, in particular, the maximum principle in Theorem \ref{sec2.6} given below, hold for this operator. The case $m=1$ (Laplacian) is included in \cite{BM6}.
Observe that $k=p+1\ge 1$ in this case.
\vsp
(iii) If $H$ is odd in $X$ i.e., $H(q,-X)=-H(q,X)$ then (\ref{sec2.4}) shows that $H(e, \lam e\otimes e+I)=-H(e, -\lam e\otimes e-I)$ and $\Lam^{\sup}=-\Lam^{\inf}<\infty$. Clearly,
$H(e,\pm e\otimes e)=0$.
\vsp
(iv) We record a simple observation. If $k_1=0$ i.e., $k=1$, then $H(e, X)= H(e/s, X)$, for any $s>0$. Thus, $H(q, X)=H(0,X)=H(X)$. \quad $\Box$
\end{rem}
\vsp

We introduce some further notation. Set $\IR^n_T=\IR^N\times (0,T)$. Let $\chi:(0,T)\rightarrow \IR$ be a bounded continuous function and, for some $m\in\IR$ (to be specified later)
$Z:[m, \infty)\rightarrow [0,\infty)$ be a non-increasing continuous function.  
For $\s\ge 0$, set
\eqRef{sec2.5}
\Pc(t,u, u_t, Du, D^2u)=H(Du, D^2u+Z(u)Du\otimes Du)+\chi(t)|Du|^\s-u_t.
\ee

We assume through out that $H$ satisfies Conditions A, B and C. 
Define
\eqRef{sec2.51}
k=k_1+1\quad\mbox{and}\quad \g=k+1=k_1+2.
\ee
Clearly, $\g\ge 2$ and if $k=1$ then $k_1=0$ and $\g=2$. Next, define
\eqRef{sec2.52}
\forall \;\s>1,\;\st=\frac{\s}{\s-1},\;\;\mbox{and,}\;\;\forall\;k>1,\;\gs=\frac{\g} {k-1}=\frac{\g}{\g-2}.
\ee 
\vsp
For a fixed $z\in \IR^n$ and $\forall x\in \IR^n$, set $r=|x-z|.$ Also, define $B^R_T=\{(x,t): |x-z|\le R,\;0<t<T\}.$ Let $\Pc$ be as defined in (\ref{sec2.5}). 
\vsp
We first state the results for $k>1$ or equivalently for $\g>2$.
\vsp
\begin{thm}\label{sec2.6}{(Maximum Principle)} Let  $0<T<\infty$, $h:\IR^n\rightarrow \IR$ be continuous with $\sup_{\IR^n} h(x)<\infty$ and, for some $m$, $Z:[m, \infty)\rightarrow [0,\infty)$ be non-increasing and continuous.
Suppose that (\ref{sec2.4}) holds, i.e., $\Ls<\infty$. Let $u\in usc(\IR^n_T),\;\inf u>m,$ solve
$$\Pc(t,u, u_t, Du, D^2u)\ge 0\;\mbox{in $\IR^n_T,$ and $u(x)\le h(x),\;\forall x\in \IR^n$}.$$

Let $\gs$ and $\st$ be as in (\ref{sec2.52}). Suppose that $\sup_{B^R_T}u(x,t)=  o( R^{\bt}),$ as $R\rightarrow \infty$.
Then the following hold.
\vsp
(a) If $0\le \s\le \g/2$ and $\bt=\gs$ then 

$$\sup_{\IR^n_T} u(x,t)\le \left\{\begin{array}{ccc}\sup_{\IR^n} h(x)+t(\sup_{[0,T]}|\chi(t)| ), & \s=0,\\ \sup_{\IR^n} h(x),& 0<\s\le \g. \end{array}\right.
$$

(b) If $\s>\g/2$ and $\bt=\st$ then 
$$\sup_{\IR^n_T} u(x,t)\le \sup_{\IR^n} h(x).\;\;\;\;\;\;\;\;\Box$$
\end{thm}
Observe that if $m=-\infty$ then the restriction $\inf u>m$ may be dropped. Also, note that if $\s=\g/2$ we get $\st=\g/(\g-2)=\gs$.
\vspp
\begin{thm}\label{sec2.7}{(Minimum Principle)} Let  $0<T<\infty$, $h:\IR^n\rightarrow \IR$ be a continuous function, with $\inf_{\IR^n} h(x)>-\infty$, and $Z:(-\infty, \infty)\rightarrow [0,\infty)$ be a non-increasing continuous function. We assume that $\Ls<\infty$.

Let $u\in lsc(\IR^n_T)$ solve
$$\Pc(t,u, u_t, Du, D^2u)\le 0,\;\;\mbox{in $\IR^n_T$ and $u(x)\ge h(x),\;\forall x\in \IR^n$}.$$

Let $\gs$ and $\st$ be as in (\ref{sec2.52}). Suppose that $\sup_{ B_T^R}(-u(x,t))=  o( R^\bt)$ as $R\rightarrow \infty$. Then the following hold.

(a) If $0\le \s\le \g/2$ and $\bt=\gs$ then 
$$\inf_{\IR^n_T} u(x,t)\ge \left\{ \begin{array}{ccc} \inf_{\IR^n} h(x)-t\left(\sup_{[0,t]}| \chi(t)| \right),& \s=0,\\ \inf_{\IR^n} h(x), & 0<\s\le \g. \end{array}\right.$$

(b) If $\s>\g/2$ and $\bt=\st$ then 
$$\inf_{\IR^n_T} u(x,t) \ge \inf_{\IR^n} h(x).\;\;\;\;\;\;\;\;\Box$$
\end{thm}
We impose no restrictions on $\Ln$ for Theorem \ref{sec2.7}.
\vspp
We now state analogous results for $k=1$ i.e, $\g=2$. See Remark \ref{sec2.40} (iv). 

The statement that, for some $s>0$, $w(r)=e^{o(r^s)}$ as $r\rightarrow \infty$, will mean
that $\log v^+=o(r^s)$ as $r\rightarrow \infty$, where $v^+=\max(v,0).$
\vsp

\begin{thm}\label{sec2.9}{(Maximum Principle)} Let  $0<T<\infty$, $h:\IR^n\rightarrow \IR$ be continuous with $\sup_{\IR^n} h(x)<\infty$. 
For some $m$, let $Z:[m, \infty)\rightarrow [0,\infty)$ be non-increasing and continuous. Suppose that (\ref{sec2.4}) holds, i.e., $\Ls<\infty$.

Let $u\in usc(\IR^n_T)$, $\inf u>m$, solve
$$H(D^2u+Z(u)Du\otimes Du)-u_t\ge 0\;\mbox{in $\IR^n_T,$ and $u(x)\le h(x),\;\forall x\in \IR^n$}.$$

Let $\st$ be as in (\ref{sec2.52}). Then the following hold
\vsp
(a) Suppose that $\s=0$. If $\sup_{B^R_T}u(x,t)=  e^{o(R^2)},$ as $R\rightarrow \infty$, then
$$u(x,t)\le \sup_{\IR^n}h(x)+\left( \sup_{(0,T)}\chi(t) \right) t,\;\;\forall (x,t)\in \IR^n_T.$$
\vsp
(b) Let $0<\s\le 1$. If $\sup_{B^R_T}u(x,t)= e^{o(R)},$ as $R\rightarrow \infty$ then
$$u(x,t)\le \sup_{\IR^n}h(x)+K(1-\s)\left( \sup_{(0,T)} \chi(t)\right),$$
where $K=K(\al, \Ls,\s, T)$.
\vsp
(c) Let $1<\s<\infty$ and assume that $\sup_{ B_T^R}u(x,t)= o(R^{\st}),$ as $R\rightarrow \infty$. Then
$$u(x,t)\le \sup_{\IR^n}h(x). \qquad \Box$$

\end{thm}
\vsp
We now present a minimum principle. Note that the condition $\Ln>-\infty$ is needed only for parts (a) and (b) of the theorem. Part (c) of the theorem holds without this restriction.
\vsp
\begin{thm}\label{sec2.10}{(Minimum Principle)} Let  $0<T<\infty$, $h:\IR^n\rightarrow \IR$ be continuous, with $\sup_{\IR^n} h(x)<\infty$, and 
$Z:(-\infty, \infty)\rightarrow [0,\infty)$ be non-increasing and continuous. We assume that $\Ls<\infty$.

Let $u\in usc(\IR^n_T)$ solve
$$H(D^2u+Z(u)Du\otimes Du)-u_t\le 0\;\mbox{in $\IR^n_T,$ and $u(x)\ge h(x),\;\forall x\in \IR^n$}.$$

Assume for parts (a) and (b) that $\Ln>-\infty$. Let $\st$ be as in (\ref{sec2.52}). Then the following hold.
\vsp
(a) Suppose that $\s=0$. If $\sup_{B^R_T}(-u(x,t))=  e^{o(R^2)},$ as $R\rightarrow \infty$, then
$$u(x,t)\ge \inf_{\IR^n}h(x)-t\left( \sup_{(0,T)}\chi(t) \right) ,\;\;\forall (x,t)\in \IR^n_T.$$
\vsp
(b) Let $0<\s\le 1$. If $\sup_{B^R_T}(-u(x,t))= e^{o(R)},$ as $R\rightarrow \infty$ then
$$u(x,t)\ge \inf_{\IR^n}h(x).$$
\vsp
(c) Let $1<\s<\infty$ and assume that $\sup_{B_T^R}(-u(x,t))= o(R^{\st}),$ as $R\rightarrow \infty$. Then
$$u(x,t)\ge \inf_{\IR^n}h(x). \qquad \Box$$
\end{thm}
\vsp
Finally, we present similar results for a class of doubly nonlinear equations of the type
$$H(Du, D^2u)-f(u)u_t=0,\;\;\mbox{in $\IR^n_T,$\; with $u(x,0)=g(x),\;\forall x\in \IR^n$.}  $$
If $k=1$, we assume that $f\equiv 1$ and the differential equation then reads
\eqRef{sec2.70}
H(D^2u)-u_t=0,\;\;\mbox{in $\IR^n_T$ with $u(x,0)=g(x),\;\forall x\in \IR^n$.}
\ee 
The above is not doubly nonlinear but is contained in our work. 

It is to be noted that the afore stated theorems are used to obtain a maximum principle for these equations. The minimum principle, however, requires a different treatment.

If $k>1$ we take $f: [0,\infty)\rightarrow[0,\infty)$ to be an increasing $C^1$ function such that
$f^{1/(k-1)}$ is concave and consider equations of the type
\eqRef{sec2.71}
H(Du,D^2u)-f(u) u_t=0,\;\;\mbox{in $\IR^n_T$ with $u(x,0)=g(x),\;\forall x\in \IR^n$,}
\ee 
where $u>0$. 

For $k>1$, let $F$ be a primitive of $f^{-1/(k-1)}$. Since $f(s)>f(0)\ge  0,\;\forall s>0$, we consider the following two situations: 
\eqRef{sec2.71.0}
\mbox{(i)}\;\; \lim_{\ve \rightarrow 0^+} F(1)-F(\ve) <\infty, \quad\mbox{and}\quad \mbox{(ii)}\;\;\lim_{\ve \rightarrow 0^+} F(1)-F(\ve) =\infty.
\ee
\vsp
We set $\chi(t)\equiv0$ in Theorems \ref{sec2.6} and \ref{sec2.10}. 
\vsp

\begin{thm}\label{sec2.8} 
Let $f: [0,\infty)\rightarrow [0,\infty)$ be a $C^1$ increasing function and $g:\IR^n\rightarrow (0,\infty)$, continuous, be such that $0<\inf_x g(x)\le \sup_x g(x)<\infty.$ Assume that
$\Ls<\infty$.

\NI (a) Maximum Principle: Let $k>1$ and $f^{1/(k-1)}$ be a concave function. Suppose that
$\phi:\IR\rightarrow [0,\infty)$ is a $C^2$ increasing function such that 
$\phi^{\prime}(\tau)= f(\phi(\tau))^{1/(k-1)}.$ Recall $\gs$ from (\ref{sec2.52}).

If $u\in usc( \overline{\IR^n_T}),\;u>0,$ solves
$$H(Du, D^2u)-f(u) u_t\ge 0,\;\mbox{in $\IR^n_T$ and $u(x,0)\le g(x),\;\forall x\in \IR^n$,}$$
and $\sup_{B^R_T} u(x,t)\le \phi( o(R^{\gs}))$, as $R\rightarrow \infty$, then
$$\sup_{ \IR^n_T}u(x,t)\le \sup_{\IR^n} g(x).$$

Let $k=1$ and $f\equiv 1$, i.e, $H(D^2u)-u_t\ge 0$. If $\sup_{B_T^R} u(x,t)\le e^{o(R^2)}$, as $R\rightarrow \infty$, then
$\sup_{ \IR^n_T}u(x,t)\le \sup_{\IR^n}g(x).$
\vsp
\NI (b) Minimum Principle: Let $k>1$, $f$ and $\phi$ be as in part (a). 

Suppose that $u\in lsc( \overline{\IR^n_T}),\;u>0,$ solves
$$H(Du, D^2u)-f(u) u_t\le 0,\;\mbox{in $\IR^n_T$ and $u(x,0)\ge g(x),\;\forall x\in \IR^n$.}$$

If condition (\ref{sec2.71.0})(i) holds, i.e, $\lim_{\ve\rightarrow 0^+}F(1)-F(\ve)<\infty$ then
$$u(x,t)\ge \inf_{\IR^n} g(x),\;\forall(x,t)\in \IR^n_T.$$

If condition (\ref{sec2.71.0})(ii) holds, i.e., $\lim_{\ve\rightarrow 0^+} F(1)-F(\ve)=\infty$, and $\inf_{B_T^R}u(x,t)\ge \phi(-o(R^{\gs}))$ as $R\rightarrow \infty$ then
$$u(x,t)\ge \inf_{\IR^n} g(x),\;\forall(x,t)\in \IR^n_T.$$
\vsp
Suppose $k=1$ and $f\equiv 1$, i.e., $H(D^2u)-u_t\le 0$. If $\inf_{B^R_T} u(x,t)\ge -e^{o(R^2)}$, as $R\rightarrow \infty$, then
$$u(x,t)\ge \inf_{\IR^n}g(x),\;\;\forall(x,t)\in \IR^n_T.$$
\end{thm}
\vspp

\section{Preliminaries}
\vsp
In this section, we present some definitions, lemmas and remarks we will use to prove the main results. Fix $z\in \IR^n$ and set $r=|x-z|,\;\forall x\in \IR^n$. A unit vector in $\IR^n$ is denoted by $e=(e_1,e_2,\cdots,e_n)$.
\vsp
We begin with an elementary remark that will be used frequently in our work.
 
\begin{rem}\label{sec3.2} Assume that $w:\IR^n\times [0,\infty)\rightarrow \IR$ is a $C^{1}$ function in $x$ and $t$ and $C^2$ in $x$ except, perhaps, at $x\ne z $. We get, for $r\ne 0$,
\bea\label{sec3.3}
&&H(Dw, D^2w+Z(w) Dw\otimes Dw)\nonumber\\
&&\quad\qquad\qquad\qquad\qquad\qquad=H\left( w_r e, \; \left( \frac{w_r}{r} \right) I +\left(w_{rr} -\frac{w_r}{r}+ (w_r)^2 Z(w) \right)e\otimes e \; \right),
\eea
where $e=(e_1,e_2,\cdots, e_n)$ with $e_i=(x-z)_i/r,\;\forall i=1,2,\cdots, n.$ Let $\kappa:(0,T)\rightarrow [0,\infty)$ be a $C^1$ function.

{\bf Case (a) ($w_r\ge 0$):} We apply Condition B, in (\ref{sec2.2}), to (\ref{sec3.3}). Factor $w_r$ from the first entry, $w_r/r$ from the second and use $k=k_1+1$ to get
\bea\label{sec3.4}
&&H(Dw, D^2w+Z(w) Dw\otimes Dw)    \nonumber \\
&&\qquad\qquad\qquad\qquad= \left( \frac{ w_r^k}{ r }\right) H\left( e, I +\left( \frac{r w_{rr}}{w_r}-1 + r w_r Z(w)\right)e\otimes e\right), \;\forall r>0.
\eea
\vsp
If $w(x,t)=\kappa(t) v(r)$, with $v^\prime(r)\ge 0$, then (\ref{sec3.4}) implies that, in $r>0$, 
\bea\label{sec3.6}
&&H(Dw, D^2w+Z(w) Dw\otimes Dw) \nonumber\\
&&\qquad\qquad=\frac{ (\kappa(t) v^\prime(r))^k}{r} H\left( e, \;I +\left( \frac{r v^{\prime\prime}(r)}{v^\prime(r)}-1 + r\kappa(t) v^{\prime}(r) Z(w) \; \right)e\otimes e\;\right).
\eea
\vsp
{\bf Case (b) ($w_r\le 0$):} Clearly, (\ref{sec3.3}) leads to
\bea\label{sec3.7}
&&H(Dw, D^2w+Z(w) Dw\otimes Dw) \nonumber\\
&&\qquad\qquad\qquad\qquad=\frac{ |w_r|^k}{ r } H\left( e, \left( 1-\frac{r w_{rr}}{w_r} + r |w_r| Z(w)\right)e\otimes e-I\right),\;\;\forall r>0.
\eea
If $w(x,t)=\kappa(t)v(r)$ and $v^{\prime}(r)\le 0$ then (\ref{sec3.7}) leads to the following analogue of (\ref{sec3.6}):
\bea\label{sec3.8}
&&H(Dw, D^2w+Z(w) Dw\otimes Dw) \nonumber\\
&&\quad\qquad=\frac{ (\kappa(t)|v^\prime(r)|)^k}{ r } H\left( e, \;\left(\; r|v^{\prime}(r)|\kappa(t)Z(w)\;+1   - \frac{rv^{\prime\prime}(r)}{v^\prime(r)}\right)e\otimes e -I \right).
\quad \Box
\eea
\end{rem}
\vsp
The following lemma was proven in \cite{BM6}. 

\begin{lem}\label{sec3.10} Let $\bt,\;\bb$ be such that $1<\bb<\bt$ and $R>0$. Fix $z\in \IR^n$, set $r=|x-z|$ and
define
$$v(r)=\int_0^{ r^\bt } \frac{1}{1+\tau^p} \;d\tau,\quad \mbox{where}\;\;p=\frac{\bt-\bb}{\bt}. $$
\vsp
Then  (i)\;$0<p<1$,\quad  (ii)\;$(1-p)\bt=\bb$, and
\ben
\mbox{ (iii)\;\;$\forall r\ge 0$},\quad \frac{ r^{\bt}  }{ 1+r^{\bt p}  }\le v(r)\le \min\left( r^{ \bt},\; \frac{ r^{\bb} }{1-p}   \right). 
\een

If $R>1$ then 
\ben
\mbox{(iv)}\;\;\frac{\bt}{2\bb}=\frac{1 }{2(1-p)}\le \frac{ v(r)-v(R) }{r^{\bb}-R^{\bb}} \le \frac{1 }{1-p}=\frac{\bt}{\bb}, \quad \forall r\ge R.
\een
Moreover, $v^{\prime}(r)=\bt r^{\bt-1}/(1+ r^{p\bt})$ implying that
\ben
&& \mbox{(v)}\;\;v^{\prime}(r)\le\bt \min\left( r^{ \bb-1},\;r^{\bt-1}\right),\;\mbox{(vi)}\;\;\frac{ (v^{\prime}(r))^k }{ r}\le \bt^k\min\left( r^{k\bt-\g},\;r^{k\bb -\g} \right),\\
&&\mbox{and (vii)}\;\;v^{\prime\prime}(r) =\bt r^{\bt-2} 
\left( \frac{ (\bt-1)+(\bb-1) r^{p\bt} }{  ( 1+r^{p\bt})^2 } \right).\
\een
\end{lem}

{\bf Comment:} Parts (iii) and (iv) of Lemma \ref{sec3.10} show that $v(r)$ grows like $r^\bt$ near $r=0$ and like $r^{\bb}$ for large values of $r$. Since $\bt\ge \bb$, one can design the function to decay fast enough at $r=0$ so as to be differentiable while its growth rate for large values of $r$ may be slower.  

\begin{proof} Parts (i)-(iii) follow quite readily. For part (iv), we take $R>1$ and write
\ben
v(r)=\int_0^{r^\bt} \left(1+\tau^p \right)^{-1} d\tau=v(R)+\int_{R^\bt}^{r^\bt} \left( 1+\tau^p \right)^{-1} d\tau
\een
We estimate $(2\tau^p)^{-1}\le (1+\tau^p)^{-1}\le \tau^{-p}$, for $\tau\ge 1$, and use this in the second integral to obtain part (iv). For part (v), note that
$1+r^{p\bt}\ge \min(1,\;r^{p\bt}).$ Using part (ii) yields the claim. Part (vi) follows by recalling that $\g=k+1=k_1+2$.

Next,
$$v^{\prime\prime}(r)=\bt\left[  \frac{ (\bt-1)r^{\bt-2} }{1+r^{p\bt} }- \frac{  p\bt\; r^{p \bt +\bt-2} } { \left( 1+r^{p\bt} \right)^2 } \right].$$
A simple calculation leads to part (vii).
\end{proof}
\vsp
The following remark is useful for the construction of the auxiliary functions. The values of $\bb$ and $\bt$, used in the remark, are motivated by the work in Sections 4 and 5.

\begin{rem}\label{sec3.11}  For Sub-Part (iv) of Part I in Section 4, we take $k>1$ (i.e, $\g>2$) and $\s>\g/2$. We set 
$$\bt=\gs=\g/ (\g-2)\quad\mbox{ and}\quad \bb=\st=\s/(\s-1).$$ 
Then $p=(\gs-\st)/\gs=(2\s-\g)/\g(\s-1)>0.$ Clearly, $0<p<1$.

We take
\ben
v(r)=\int_0^{ r^{\gs}} \frac{1}{1+\tau^p}\;d\tau,\quad\mbox{where}\; p=1-\frac{\st}{\gs}=\frac{2\s-\g}{\g(\s-1)}.
\een

From Lemma \ref{sec3.10}, (i) $0<p<1$, \;\;(ii) $(1-p)\gs=\st$,
\ben
&&\mbox{(iii) for $r\ge 0$,}\;\;\; \frac{ r^{\gs} }{ 1+r^{\gs p} }\le v(r)\le \min\left( r^{ \gs },\; \frac{\gs r^{\st} }{\st} \right), \\
&&\mbox{(iv) for any $R>1$,}\;\;\frac{\gs}{2\st}=\frac{1 }{2(1-p)}\le \frac{v(r)-v(R)}{r^{\st}-R^{\st}}\le \frac{1 }{1-p}=\frac{\gs}{\st}, \;\; \forall r\ge R.
\een
Moreover, $v^{\prime}(r)=\gs r^{\gs-1}/(1+ r^{p\gs}),$
\ben
&&\mbox{(v)}\;\;v^{\prime}(r)\le \gs \min\left(  r^{ \st-1},\;r^{\gs-1}\right),\;\;\mbox{(vi)}\;\frac{ (v^{\prime}(r))^k }{ r}\le {\gs}^k\min\left( r^{k\gs-\g},\;r^{k\st -\g} \right),\\\ &&\mbox{and (vii)}\;\;v^{\prime\prime}(r) =\gs r^{\gs-2} 
\left( \frac{ (\gs-1)+(\st-1) r^{p\gs} }{  ( 1+r^{\gs p})^2 } \right).\quad \Box
\een
\end{rem}
\vsp
\begin{rem}\label{sec3.201} The super-solutions and sub-solutions make use of functions that involve a $C^1$ function of $t$ and a $C^{1,\al}$ (for some $\al>0$) function of $v(r)$. See the functions discussed in Remark \ref{sec3.11}. The calculations done in the remark hold in the sense of viscosity at $r=0.$ The verification can be found in \cite{BM6}. 
\quad $\Box$
\end{rem}
We recall a comparison principle needed for our work, see \cite{CIL}. See also \cite{BM5} and \cite{BM6}.
\vsp
Let $F:\IR^+\times \IR\times \IR^n\times S^n\rightarrow\IR$ be continuous. Suppose that $F$ satisfies $\forall X,\;Y\in S^n$, with $X\le Y$, that
\bea\label{sec3.17}
\mbox{$F(t, r_1, q, X)\le F(t, r_2, q, Y),$ $\forall (t,q)\in \IR^+\times \IR^n$ and $r_1\ge r_2$.}
\eea
In this work, $F(t, r, q, X)=H(q, X+Z(r) q\otimes q)+\chi(t)|q|^\s,$ where $Z$ is a non-increasing continuous function, $\s\ge 0$ and $H$ satisfies Conditions A, B and C.
 \vsp
Let $\Om\subset \IR^n$ be a bounded domain, $\Om_T=\Om\times (0,T)$ and $P_T$ be the parabolic boundary of $\Om_T$.

\begin{lem}\label{sec3.18}{(Comparison principle)} Let $F$ satisfy (\ref{sec3.17}) and $\hat{f}:\IR^+\rightarrow \IR^+$ be a bounded continuous function. For some $m$, let 
$Z:[m ,\infty)\rightarrow \IR$ be a non-increasing continuous function.

Let $u\in usc(\Om_T\cup P_T)$ and $v\in lsc(\Om_T\cup P_T)$ such that 
$\inf(\inf u, \inf v)>m$. Suppose that $u$ and $v$ solve 
\ben
&& F(t,u, Du, D^2u+Z(u)Du\otimes Du)- \hat{f}(t)u_t\ge 0,\\
&&\mbox{and}\quad F(t,v, Dv, D^2v+Z(v)Dv\otimes Dv)-  \hat{f}(t)v_t\le 0,\quad \mbox{in $\Om_T$.}
\een
If $\sup_{P_T}v<\infty$ and $u\le v$ on $P_T$ then $u\le v$ in $\Om_T$. \quad $\Box$
\end{lem}

Next, we discuss a change of variables that is used in the proof of Theorem \ref{sec2.8} for doubly nonlinear equations of the kind
\eqRef{sec3.190}
H(Du, D^2u)-f(u) u_t=0,\;\;\mbox{in $\IR^n_T$, $u>0$, with $u(x,0)=g(x),\;\forall x\in \IR^n$.   }
\ee

\begin{rem}\label{sec3.19} Let $f:[0,\infty)\rightarrow [0,\infty)$ be an increasing $C^1$ function. Suppose that $k>1$ and 
$f^{1/(k-1)}$ is concave. 

Let $I$ be either $[0,\infty)$ or $(-\infty, \infty)$, see (b) and (c) below. We select 
$\phi: I\rightarrow [0,\infty)$, an increasing $C^2$ function, such that 
$$\phi^{\prime}(\tau)= f(\phi(\tau))^{1/(k-1)},\;\;\forall \tau \in I,\;\;\mbox{or}\;\;\int ^{\phi(\tau)}_{\phi(\tau_0)} f^{-1/(k-1)}(\tht) \; d\tht=\tau-\tau_0.$$ 
We define the change of variable $u=\phi(v)$ by
\eqRef{sec2.71}
v(u)-v(u_0)=\phi^{-1}(u)=\int^u_{u_0} f^{-1/(k-1)}(\tht)\;d\tht,\;\;u\ge u_0,
\ee
for some $u_0\ge 0$.

We discuss some examples. Let $\al>0$, $a\ge 0$ and $f(s)=(s+a)^\al,\;\forall s\ge 0.$  Then $f(s)^{1/(k-1)}$ is concave if $\al\le k-1$. Set $c_k=(k-1-\al)/(k-1).$ We may take $u_0=0$ in (\ref{sec2.71}), we get that
$$u=\phi(v)=\left\{ \begin{array}{ccc} \left[  c_kv+a^{c_k})\right]^{1/c_k }-a ,& 0<\al<k-1,\;a\ge 0,\\ a e^v-a, & \al =k-1,\;a>0. \end{array}\right.$$
See also part (b) below.

If $a=0$, take $f(s)=s^{k-1}$ then $u=b e^v$ for any $b>0$. But, $u_0\ne 0$, see part (c).
\vsp
We make some observations about (\ref{sec2.71}).

{\bf (a)} It is clear that $v$ is an increasing concave function of $u$. The concavity follows since $f$ is non-decreasing. Since $v$ is increasing, $u$ is a convex function of $v$.

{\bf (b)} If the integral in (\ref{sec2.71}) is convergent for $u_0=0$ we then define 
$$v=\phi^{-1}(u)=\int^u_0 f^{-1/(k-1)}(\tht)\;d\tht.$$
Thus, $v(0)=0$ and $v>0$. 

We choose $I=[0,\infty)$ and $\phi:[0,\infty)\rightarrow [0,\infty)$.
This applies to examples like
$$f(s)=\left\{ \begin{array}{ccc} s^\al, & 0\le \al<k-1,\\ (s+a)^\al, & 0\le \al \le k-1,\end{array}\right.$$
where $a>0$.
\vsp
{\bf (c)} If the integral in (\ref{sec2.71}) is divergent for $u_0=0$ then $v(u_0)\rightarrow -\infty$ as $u_0\rightarrow 0^+$.
In this case, we select a primitive
$$v=\phi^{-1}(u)=\int^u f^{-1/(k-1)}(\tht)\;d\tht.$$
We choose $I=(-\infty, \infty)$ and $\phi:(-\infty, \infty)\rightarrow (0,\infty)$. 
This includes examples such as $f(s)=s^{k-1},\;(s+\log(s+1))^{k-1}$ etc.
\vsp
{\bf (d)} We show that in parts (b) and (c), $v\rightarrow \infty$ if $u\rightarrow \infty$. Set $\nu(s)=f^{1/(k-1)}(s)$. Since $\nu(s)$ is concave in $(0,\infty)$, it is clear that, 
for a fixed $\ve>0$, 
$$\nu(s)\le \nu(\ve)+(s-\ve) \nu^{\prime}(\ve),\;\;s\ge \ve.$$
Using (\ref{sec2.71}), we get that
$$v(u)=v(\ve)+ \int_\ve^u  \frac{1}{\nu(s)}ds\ge v(\ve)+ \int_\ve^u \frac{1}{ \nu(\ve)+(s-\ve) \nu^{\prime}(\ve) } ds.$$
The claim holds.
\vsp 
{\bf (e)} It is clear from (\ref{sec2.71}) that 
$$\frac{\phi^{\prime\prime}(v)}{\phi^{\prime}(v)}= \left.\left( \frac{ d}{ds} f^{1/(k-1)}(s) \right) \right|_{\phi(v)},$$
and $\phi^{\prime\prime}(v)/\phi^{\prime}(v)$ is non-increasing in $v$ since $f^{1/(k-1)}$ is concave and $\phi(v)$ is increasing in $v$.

Suppose that there are constants $0<\om_1\le \om_2<\infty$ such that
\begin{equation}\label{eq:add}\om_1\le \phi^{\prime\prime}(v)/ \phi^{\prime}(v) \le \om_2.\end{equation}
Integrating from $s=0$ to any $s>0$, we get that,
$$\om_1 s\le f^{1/(k-1)}(s)-f^{1/(k-1)}(0)\le \om_2 s,\;\;\forall s\ge 0.$$
Since $f(0)\ge 0$, we get that, for some $\om\ge 0$, $(\om_1s+\om)^{k-1}\le f(s)\le (\om_2 s+\om)^{k-1},\;\forall s\ge 0$. 

If $\om>0$ then we use $v$ as in part (b). If $\om=0$ then we use part (c). 
\vsp
{\bf (f)} The change of variable $u=\phi(v)$, as given by (\ref{sec2.71}), transforms (\ref{sec3.190}) into 
\ben
H(Dv, D^2v+Z(v) Dv\otimes Dv)-v_t=0\;\;\mbox{in $\IR^n_T$ with $v(x,0)=\phi^{-1}(g(x)),\;\forall x\in \IR^n$,}
\een
where $Z(v)=\phi^{\prime\prime}(v)/\phi^{\prime}(v)$, see Lemma 2.3 in \cite{BM5}. By part (e), $Z(v)$ is non-increasing in $v$ and the domain of $Z$ contains either $(0,\infty)$ or $(-\infty, \infty)$. $\Box$
\end{rem}

We now state a comparison principle for doubly nonlinear equations.

\begin{lem}\label{sec3.20} Let $T>0$ and $\Om\subset \IR^n$ be a bounded domain. Suppose that $k>1$ and $f:[0,\infty)\rightarrow [0,\infty)$ is a non-decreasing $C^1$ function such that $f^{1/(k-1)}$ is concave. Set $\Om_T=\Om\times (0,T)$ and $P_T$ to be the parabolic boundary of $\Om_T$.

Let $u \in usc(\Om_T)$, $v\in lsc(\Om_T)$ and $u>0$ and $v>0$. Suppose that 
\ben
H(Du, D^2u)-f(u)u_t\ge 0,\;\;\mbox{in $\Om_T$, and}\;\;H(Dv, D^2v)-f(v) v_t\le 0,\;\;\mbox{in $\Om_T$,}
\een
where $H$ satisfies conditions $A,\;B$ and $C$. 

If $u\le v$ on $P_T$ then $u\le v$ in $\Om_T$.
\end{lem}
\begin{proof} We employ Lemma \ref{sec3.18} and Remark \ref{sec3.19}. Let $u$ and $v$ be as in the statement of the theorem. 
Set
$$F(\hat{s},s)=\int_s^{\hat{s}} f^{-1/(k-1)}(\tht) d\tht,\;\;\forall \hat{s}\ge s\ge 0.$$
We define $F(\hat u,0)=\lim_{s\rightarrow 0^+}F(\hat u,s)$, if it exists.

{\bf (i)} Suppose that $F(1,0)<\infty$ then we define
$$\bar{u}=\phi^{-1}(u)=F(u, 0)\;\;\mbox{and}\;\;\bar{v}=\phi^{-1}(v)=F(v,0).$$
By parts (a) and (b) of Remark \ref{sec3.19}, $\bar{u}>0$ and $\bar{v}>0$. Also, by part (f) of Remark \ref{sec3.19},  
\ben
H(D\bar{u}, D^2\bar{u}+Z(\bar{u}) D\bar{u}\otimes D\bar{u})-\bar{u}_t\ge 0\;\;\mbox{and}\;\; H(D\bar{v}, D^2\bar{v}+Z(\bar{v}) D\bar{v}\otimes D\bar{v})-\bar{v}_t\le 0,
\een
in $\Om_T$, where $Z(s)=\phi^{\prime\prime}(s)/\phi^{\prime}(s)$ is non-increasing in $s$. Note that the domain of $Z$ contains $(0,\infty)$. Using Lemma \ref{sec3.18}, $\bar{u}\le \bar{v}$ in $\Om_T$ thus implying that $u\le v$ in $\Om_T$.
\vsp
{\bf (ii)} Suppose now that $F(1, 0)$ is divergent, see part (c) of Remark \ref{sec3.19}. Fix a primitive 
$$F(s)=\int^s f^{-1/(k-1)}(\tht) d\tht,\;\;s>0.$$
Define 
$\bar{u}=\phi^{-1}(u)=F(u)\;\;\mbox{and}\;\;\bar{v}=\phi^{-1}(v)=F(v).$
Then $-\infty<\bar{u},\;\bar{v}<\infty$ and by parts (e) and (f) of Remark \ref{sec3.19}, we get in $\Om_T$,
\ben
H(D\bar{u}, D^2\bar{u}+Z(\bar{u}) D\bar{u}\otimes D\bar{u})-\bar{u}_t\ge 0\;\;\mbox{and}\;\; H(D\bar{v}, D^2\bar{v}+Z(\bar{v}) D\bar{v}\otimes D\bar{v})-\bar{v}_t\le 0,
\een
where the domain of $Z$ is $(-\infty,\infty)$. Using Lemma \ref{sec3.18}, $\bar{u}\le \bar{v}$ in $\Om_T$ thus implying that $u\le v$ in $\Om_T$.
\end{proof}

\section{ \bf Super-solutions}

\vsp
In this section, we construct super-solutions of (\ref{sec1.3}) and these are used to prove Theorems \ref{sec2.6}, \ref{sec2.9} and \ref{sec2.8}. We have divided our work into two parts. Part I addresses the case $k>1$(or $\g>2$) and Part II discusses the case $k=1$ or $\g=2$. In each part, the work is further sub-divided to address various situations based on the values of $\s$. Since the auxiliary functions are non-negative, we assume that the domain of $Z$ is at least $(0,\infty)$, see discussion below. 

Part I has four sub-parts:
(i) $\s=0$, (ii) $0<\s<\g/2$, (iii) $\s=\g/2$ and (iv) $\s>\g/2$, and Part II has three sub-parts: (i) $0\le \s\le 1$, (ii) $1<\s\le 2$, and (iii) $\s>2$.

We recall from (\ref{sec2.5}) that 
\eqRef{sec4.0}
\Pc(t,w,w_t,Dw,D^2w):=H(Dw, D^2w+Z(w)Dw\otimes Dw)+\chi(t)|Dw|^\s-w_t, 
\ee
where $\s\ge 0$, and $Z(s)\ge 0$ and is a non-increasing continuous function of $s$. 

Let $m<\min (0, \inf_{\IR^n} h,\inf_{\IR^n_T}u )$, where $h$ is the initial data in (\ref{sec1.3}) and $u$ is the given sub-solution. We assume that the domain of $Z$ is at least $[m,\infty)$.

Recall from (\ref{sec2.3}) and (\ref{sec2.4}) that $\Ls=\sup_{\lam} \left( \max_{|e|=1}H(e, I+\lam e\otimes e) \right).$ We set 
\eqRef{sec4.1}
\al=\sup_{[0,T]} |\chi(t)| \;\;\mbox{and}\;\; M=\max\left( \Ls, 1\right).
\ee

We also recall from (\ref{sec2.51}) and (\ref{sec2.52}) that
$$k=k_1+1,\;\; \g=k_1+2=k+1,\;\;\g\ge 2\;\;\mbox{and}\;\; \gs=\frac{\g}{\g-2}\;\;\mbox{if $\g>2$.}$$
Moreover, $\g=2$ if and only if $k=1$($k_1=0$). 
\vsp
\NI{\bf Super-solutions:}
\vsp
\NI {\bf Part I ($k>1$):} Since $\g>2$, we see that
\eqRef{sec4.2}
\gs-1=\frac{2}{\g-2},\quad \gs-2=\frac{4-\g}{\g-2}>-1\quad\mbox{and}\quad k\gs-\g=\frac{\g}{\g-2}=\gs.
\ee
\vsp
We start with the case $0\le \s\le \g/2$ and first carry out some calculations that will hold for the entire interval $[0,\g/2]$. We will then discuss the cases $\s=0$, $0<\s<\g/2$ and
$\s=\g/2$ separately. 

Let $z\in \IR^n$ be fixed, set $r=|x-z|,\;\forall x\in \IR^n$, and define
\bea\label{sec4.4.0}w(x,t)=at+ b(1+t) v(r), \;v^{\prime}(r)\ge 0,\;\;\forall(r,t)\in \IR^n_T,\eea
where $a\ge 0$ and $0<b\le 1$ are to be determined. We do this in each of the three cases listed above and also calculate $\lim_{b\rightarrow 0} a,$ wherever it is meaningful.

Using (\ref{sec3.6}), (\ref{sec4.0}) and (\ref{sec4.1}), we get
\bea\label{sec4.4}
&&\Pc(t,w,w_t,Dw,D^2w)\nonumber \\
&&\quad\qquad=\frac{ [ b(1+t) v^\prime(r) ]^k}{r} H\left( e, \;I +\left( \frac{r v^{\prime\prime}(r)}{v^\prime(r)}-1 + b(1+t) r v^{\prime}(r) Z(w) \; \right)e\otimes e\;\right) \nonumber\\
&&\quad\qquad\qquad\qquad+\chi(t) [ b(1+t) v^\prime(r) ]^\s-a-bv(r)   \nonumber \\
&&\quad\qquad\le \frac{M [ b(1+T) ]^k v^\prime(r)^k}{r}  + \al \left [ b(1+T) \right ]^\s  (v^\prime(r))^\s-a-bv(r).
\eea
We use the above inequality in both Parts I and II. 
\vsp
For Part I, we take $v(r)=r^{\gs}$. Using (\ref{sec4.2}) and $k=\g-1$ in (\ref{sec4.4}), we find that
\bea\label{sec4.5.1}
&&\Pc(t,w,w_t, Dw, D^2w)\nonumber\\
&&\qquad\qquad \le  M [ b\gs(1+T) ]^k \frac{(r^{\gs-1} )^k }{r} + \al\left [ b \g^* (1+T)\right ]^\s  \left( r^{\g^*-1}\right)^\s-a- br^{\g^*}\nonumber \\
&&\qquad\qquad\le M\left[ \g^* (1+T)\right ]^k ( b^k r^{\gs} )
+ \al \left [ \g^* (1+T) \right ]^\s  ( b^\s r^{ 2\s/(\g-2)} )- a- ( br^{\g^*} ).
\eea

In order to write more compactly, we set 
$$E=M\left [\g^*  (1+T)\right ]^k\quad \mbox{and}\quad F= \left[  \g^* (1+T) \right ]^\s.$$
Thus, (\ref{sec4.5.1}) reads
\bea\label{sec4.40}
 \Pc(t,w,w_t,Dw,D^2w)\le E  (b^k r^{\gs})+\al F ( b^\s r^{2\s/(\g-2)}) -a- (br^{\gs}).
 \eea
 \vsp 
{\bf Sub-Part (i) ($\s=0$):} Taking $\s=0$ in (\ref{sec4.40}), we get that $F=1$ and
$$\Pc(t,w,w_t,Dw,D^2w)\le b(Eb^{k-1}-1) r^{\gs}+\al-a.$$
Select $a=\al$ and $0<b<\min(1,E^{1-k})$. Clearly, $w(x,t)$ is a super-solution in $\IR^n_T$ and
\eqRef{sec4.41}
w(x,t)=\al t+b(1+t)r^{\gs}. \quad\qquad \Box
\ee
\vsp
{\bf Sub-Part (ii) ($0<\s<\g/2$):} Since $\gs=\g/(\g-2)$, (\ref{sec4.40}) yields that
\bea\label{sec4.42}
\Pc(t,w,w_t,Dw,D^2w)&\le& E b^k r^{\gs}-br^{\gs}+\al F b^\s r^{2\s/(\g-2)}-a   \nonumber\\
&=&b r^{\gs} \left(  Eb^{k-1}-1+ \frac{ \al F b^{\s-1} }{ r^{(\g-2\s)/(\g-2) } } \right)-a. 
\eea
We choose 
\eqRef{sec4.43}
\left\{ \begin{array}{lcr} 0<b^{k-1}<\min\left(1, \;\; (4E)^{-1}\right),\;\;R= \left( 4 \al F b^{\s-1}\right)^{(\g-2)/(\g-2\s)},\\ 
\mbox{and}\;\;a= E b^k R^{\gs}+\al F b^\s R^{2\s/(\g-2)}.
\end{array}\right. 
\ee
The choice for $a$ shows that $w$ is a super-solution in $B_R(z)\times (0,T)$. In $r\ge R$, using $0<\s<\g/2$ and 
the selections for $b$ and $R$, stated in (\ref{sec4.43}), in (\ref{sec4.42}), we get
$$Eb^{k-1}-1+ \frac{\al F b^{\s-1} }{ r^{(\g-2\s)/(\g-2) } }\le -\frac{3}{4}+ \frac{ \al F b^{\s-1} }{ R^{(\g-2\s)/(\g-2) } }= -\frac{3}{4}+\frac{1}{4}=-\frac{1}{2}.$$
Thus, $w$ is a super-solution in $\IR^n_T$ for any $a$ and $b>0$ satisfying the requirement in (\ref{sec4.43}).
\vsp
We now evaluate $\lim_{b \rightarrow 0} a$. If $\s\ge 1$, it is clear from (\ref{sec4.43}) that $\lim_{b\rightarrow 0} a=0$.
Let $0<\s<1$. Using (\ref{sec4.43}), $\gs=\g/(\g-2)$ and $k=\g-1$, we obtain that, for some $K_1$ and $K_2$, independent of $b$,
\ben
b^k R^{\gs}&=&K_1 b^{\g-1} \left(  b^{ (\s-1) (\g-2)/(\g-2\s)} \right)^{\g/(\g-2)} =K_1 b^{(\g-\s)(\g-2)/(\g-2\s)},\\
\mbox{and}\;\;b^\s R^{2\s/(\g-2)}&=&K_2 b^\s \left( b^{(\s-1)(\g-2)/(\g-2\s)} \right)^{2\s/(\g-2)}=K_1b^{\s (\g-2)/(\g-2\s)}.
\een
It is clear that 
\eqRef{sec4.44}
 \lim_{b \rightarrow 0} a=0. \quad\qquad \Box
 \ee
\vsp
{\bf Sub-Part (iii) ($\s=\g/2$):} We modify $w$ as follows. Take
\eqRef{sec4.5}
w(x,t)=b (t+1) r^{\g^*},
\ee
where $b>0$ is to be determined. Note that
$$\gs=\frac{\g}{\g-2}=\frac{2\s}{ \g-2 }.$$

Taking $a=0$ in (\ref{sec4.40}) and observing that $k>1$ and $\g>2$, we get
\ben
 &&\Pc(t,w,w_t,Dw,D^2w) \\
 &&\quad\qquad\qquad\qquad \le E b^k r^{\gs}+\al F b^\s r^{2\s/(\g-2)}-br^{\gs}=Eb^kr^{\gs}+\al F b^{\g/2} r^{\gs}-b r^{\gs}\nonumber\\
 &&\quad\qquad\qquad\qquad =b r^{\gs} \left( Eb^{k-1}+\al F b^{(\g-2)/2}-1\right)\le 0,\nonumber
\een
if $0<b\le b_0$, for some $b_0=b_0(\al,k,\g, E, F)$ chosen small enough. Thus,
\eqRef{sec4.50}
w(x,t)=b(1+t) r^{\gs},\;\;\forall 0<b\le b_0,
\ee
is a super-solution in $\IR^n_T$.  $\Box$
\vsp
{\bf Sub-part (iv) ($\s>\g/2$):}  We use Remark \ref{sec3.11} and take
\eqRef{sec4.6}
w(x,t)=at+b(1+t)v(r),
\ee
where
$$v(r)=\int_0^{r^{\g^*} } \frac{1}{1+\tau^p}\;d\tau,\;\;\; p=\frac{\gs-\st}{\gs},\;\;\g^*=\frac{\g}{\g-2}\;\;\mbox{and}\;\;\st=\frac{\s}{\s-1}.$$
Here $a>0$ and $0<b\le 1$ are to be determined. Note that $v(r)$ grows like $r^{\gs}$ near $r=0$ and like $r^{\st}$ for large $r$. 

Recall (\ref{sec4.4}) i.e.,
\eqRef{sec4.7}
\Pc(t,w,w_t,Dw,D^2w)\le \frac{M [ b(1+T)]^k  v^\prime(r)^k}{r} + \al\left[ b(1+T) \right]^\s v^{\prime}(r)^\s-a-bv(r).
\ee
\vsp
We use parts (ii)-(viii) of Remark \ref{sec3.11}, $k=\g-1$ and $(\st-1)\s=\st$. Note that
\ben
&&( v^{\prime}(r) )^\s\le (\gs)^\s \min\left( r^{\st-1},\;r^{\gs-1} \right)^\s=(\gs)^\s \min \left( r^{\st},\;r^{2\s/(\g-2)} \right),\\
&&\mbox{and}\quad \frac{ (v^{\prime}(r))^k}{r}\le \min (\gs)^k \left( r^{k\st-\g},\;r^{k\gs-\g}\right)=(\gs)^k \min \left( r^{(\g-\s)/(\s-1)},\;r^{\gs}\right).
\een
Using the above in (\ref{sec4.7}) and recalling the definitions of $E$, $F$ (see the line following (\ref{sec4.5.1})) we get that
\bea\label{sec4.80}
\Pc(t, w,w_t, Dw, D^2w) \le Eb^k r^{(\g-\s)/(\s-1)} + \al Fb^\s r^{\st}-a-b v(r).
\eea
A lower bound for $v(r)$ is obtained by setting $R=1$ in Remark \ref{sec3.11}(iv) and ignoring $v(1)$. Taking $r\ge 1$, (\ref{sec4.80}) yields that
\bea\label{sec4.8}
\Pc(t, w,w_t, Dw, D^2w) &\le& Eb^k r^{(\g-\s)/(\s-1)} + \al Fb^\s r^{\st}-a-\frac{b \gs  \left( r^{\st} - 1 \right)}{2\st}   \nonumber\\
&=&Eb^k r^{(\g-\s)/(\s-1)} + \al Fb^\s r^{\st}+\frac{b\gs }{2\st}-a- \frac{ b\gs r^{\st}}{2\st},
\eea
where we have used that $1-p=\st/\gs$.

We select
\eqRef{sec4.85}
a= Eb^k+\al F b^\s +\frac{b\gs }{\st},
\ee
From (\ref{sec4.8}) and (\ref{sec4.85}), it follows that  
$w$ is a super-solution in $B_1(o)\times [0,T]$.

Since $r^{(\g-\s)/(\s-1)}\le r^{\st}$, in $r\ge 1$, using (\ref{sec4.85}) in (\ref{sec4.8}) implies that
\ben
\Pc(t, w_t, Dw, D^2w)& \le &Eb^k r^{\st}+ \al Fb^\s r^{\st}+\frac{b\gs }{2\st}-a- \frac{ b r^{\st}}{2\st}\\
&\le &b r^{\st} \left( E b^{k-1}+\al F b^{\s-1} -\frac{\gs}{2\st} \right)\le 0,
\een
if we select $0<b\le b_0$, where $b_0$ depends only on $\al,\g,\;\s,\;E$ and $F$, and is chosen small enough. 

Thus, $w$ is super-solution in $\IR^n_T$ and 
\eqRef{sec4.9}
\lim_{b\rightarrow 0} a=0.\qquad \Box
\ee
 \vsp
 {\bf Part II ($k=1$):} In this case, $\g=2$ and $k_1=0$. 
 
 By Remark \ref{sec2.40}(iv), $H(q, X)=H(X),\;\forall (q,X)\in \IR^n\times S^{n\times n}$. Thus, we work with
 $$H(D^2u+Z(u)Du\otimes Du)+\chi(t)|Du|^\s-u_t\ge 0,\;\;\mbox{in $\IR^n_T$ with $u(x,0)\le h(x),\;\forall x\in \IR^n$.}$$
 
We treat separately the three possibilities: (i) $0\le \s\le 1$, (ii) $1<\s\le 2$ and (iii) $2<\s<\infty$.
 
\vsp
{\bf Sub-Part (i) ($0\le \s\le1$):}  Take
\eqRef{sec4.11}
w(x,t)=at+b(1+t)v(r),\quad \forall (x,t)\in \IR^n_T,
\ee
where $a\ge 0$ and $0<b\le1$ are to be determined. 
\vsp
{\bf (a) ($\s=0$):} We choose 
$$v(r)=e^{cr^2}.$$
where $c>0$ is to be determined. We note the following elementary facts.
\ben
v^{\prime}(r)=2cre^{cr^2},\;\; \frac{v^{\prime}(r)}{r}=2c e^{cr^2},\;\;\;\mbox{and}\;\; \frac{ r v^{\prime\prime}(r)}{v^{\prime}(r)}=1+2cr^2.
\een
Using these in (\ref{sec4.4}) and using $\s=0$, we get
\ben
\Pc(t,w,w_t,Dw,D^2w)& \le& b(1+T)M  \left(  \frac{v^\prime(r)}{r} \right)+\al -a-bv(r)\\
&=& 2bc(1+T)M e^{cr^2}+\al -a-b e^{cr^2}.
\een
Set $a=\al$, $\bar{E}=2(1+T)M$ and $c=1/\bar{E}$ to obtain $\Pc(t,w,w_t,Dw,D^2w) \le 0$ in $\IR^n_T$.

Thus, 
$$ w(x,t)=\al t+ b(1+t) e^{r^2/\bar{E}},\quad \forall (x,t)\in \IR^n_T,$$
is a super-solution in $\IR^n_T$ for any $b>0$. Moreover,
\eqRef{sec4.131}
\lim_{b\rightarrow 0} w(x,t)=\al t. \qquad \Box
\ee

\vsp
{\bf (b) ($0<\s\le1$):}  For $a>0$, $0<b\le 1$ and $c>0$ (to be determined), we define
\eqRef{sec4.20}
w(x,t)=at+b(1+t) v(r),\;\;\mbox{in $\IR^n_T$},\;\; \mbox{where}\;\; v(r)=e^{cr}-(1+cr).
\ee

Thus,
$$v^{\prime}(r) =c \left( e^{cr}-1\right),\quad c^2\le \frac{v^{\prime}(r)}{r}\le c^2 e^{cr},\quad \mbox{and}\quad 1\le \frac{r v^{\prime\prime}(r)}{v^{\prime}(r)}\le \frac{ e{\max(1, cr)}}{e-1}.$$
In the last estimate, for $0 < \theta < 1$ we used that $\theta e^\theta/(e^\theta-1)$ is increasing  and for $1<\theta,$ we used
that $e^\theta/(e^\theta-1)$ is decreasing.

Applying the above to (\ref{sec4.4}), we obtain
\bea\label{sec4.200}
&&\Pc(t,w,w_t,Dw,D^2w)\nonumber \\
&&\quad\qquad\le b(1+T)M\left( \frac{v^\prime(r)}{r} \right)+ \al \left[ b(1+T) v^\prime(r) \right]^\s-a-bv(r)  \\
&&\quad\qquad\le bc^2(1+T) M e^{cr}+\al \left[ bc (1+T) \left( e^{cr}-1\right) \right]^\s -a- b\left( e^{cr}-1-cr\right).\nonumber
\eea
Set $\bar{E}=(1+T)M$ and $\bar{F}=\al (1+T)^\s$. A rearrangement of the above leads to 
\eqRef{sec4.141}
\Pc(t, w, w_t, Dw, D^2w)\le b(1+cr)+ \left(c^2\bar{E}\right) b e^{cr}+\left( c^\s \bar{F}\right) (be^{cr})^\s-b e^{cr}-a.
\ee  

Applying Young's inequality $(be^{cr})^\s\le (1-\s)+\s b e^{cr}$, (\ref{sec4.141}) implies that
\ben
\Pc(t, w, w_t, Dw, D^2w)&\le& b(1+cr)+e^{cr}\left( c^2\bar{E}+\s c^\s \bar{F}-1\right)+(1-\s)c^\s\bar{F}-a\\
&\le& \left[ (1-\s)c^\s \bar{F}-a \right]+b\left[ (1+cr)+  e^{cr}\left( c^2\bar{E}  +\s c^\s \bar{F}-1 \right) \right].
\een

Select $c>0$ such that $c^2\bar{E}+\s c^\s \bar{F}=1-\ve$, for a fixed small $0<\ve<1$. Hence,
\eqRef{sec4.201}
\Pc(t, w, w_t, Dw, D^2w)\le\left[ (1-\s)c^\s \bar{F}-a \right]+b\left[ (1+cr)-\ve e^{cr} \right].
\ee
\vsp
The maximum of the function $1+cr-\ve e^{cr}$ occurs at $r_0=c^{-1} \log(1/\ve)$ and the maximum value is $\log(1/\ve)$. Select
$$a= b \log(1/\ve)+(1-\s)c^\s \bar{F}.$$

Using the choice for $a$ in (\ref{sec4.201}), we get that $\Pc(t,w_t, Dw, D^2w)\le 0$, in $\IR^n_T$.
Thus, $w$ is a super-solution in $\IR^n_T$ and 
\eqRef{sec4.202}
\lim_{b \rightarrow 0} a=(1-\s)c^{\s}\bar{F}, \;\;0<\s\le 1.\qquad \Box
\ee
\vsp
Observe that if $\s=0$ then $\lim_{b \rightarrow 0}a=\al$. While this agrees with part (a), the growth rate allowed in part (a) is greater. Also, if we take $\s=1$, $\lim_{b\rightarrow 0} a=0$. 

\vsp

{\bf Sub-Part (iii) ($1<\s\le 2$):} For $a>0$ and $0<b\le 1$ (to be determined), we select
\eqRef{sec4.30}
w(x,t)=at+b(1+t) r^{\st},\;\;\forall(x,t)\in \IR^n_T,\;\;\mbox{where}\;\; \st=\frac{\s}{\s-1}.
\ee
\vsp
Note that $\st\ge 2$. Setting $v(r)=r^{\st}$, we find that
$$\frac{v^{\prime}(r)}{r}=\st r^{\st-2}=\st r^{(2-\s)/(\s-1)},\quad v^{\prime}(r)^\s=(\st)^\s r^{\st}\quad\mbox{and}\quad \frac{rv^{\prime\prime}(r)}{v^{\prime}(r)}=\st-1. $$
Using the above in (\ref{sec4.200}) or (\ref{sec4.4}) and recalling the definitions of $\bar{E}$ and $\bar{F}$ (see Sub-Part (ii)) we obtain that
\bea\label{sec4.31}
\Pc(t,w,w_t,Dw, D^2w)&\le& \bar{E} \left(\frac{b v^\prime(r)}{r}\right)+ \bar{F}  (bv^\prime(r))^\s-a-bv(r)   \nonumber\\
&=& (\st \bar{E} )b r^{\st-2}+( {\st}^{\s} \bar{F} ) b^{\s} r^{\st}-a-br^{\st}.
\eea

Choose
\ben
R=\sqrt{4\st \bar{E}},\;\; 0<b<  \left( \frac{1}{4 {\st}^\s \bar{F} } \right)^{1/(\s-1)} \;\;\mbox{and}\;\;a=(\st \bar{E}) bR^{\st-2}+ ( {\st}^{\s} \bar{F} )b^{\s} R^{\st}. 
\een
Employing the above values in (\ref{sec4.31}) and noting that $\st\ge 2$, we see that $w$ is super-solution in $[0,R]\times [0,T].$ In $r\ge R$, 
\ben
\Pc(t,w,w_t,Dw, D^2w)&\le&  (\st \bar{E}) br^{\st-2}+( {\st}^\s \bar{F} )b^\s r^{\st}-a-br^{\st}\\
&=&b r^{\st}\left( \frac{\st \bar{E}}{r^2}+ ({\st}^{\s} \bar{F} ) b^{\s-1} -1 \right)-a.
\een
Using the values of $R$ and $b$, it is clear that $w$ is super-solution in $\IR^n_T$. Moreover,
\eqRef{sec4.32}
\lim_{b\rightarrow 0} a=0.\qquad \Box
\ee
\vsp
{\bf Sub-Part (iv) $2<\s<\infty$:} We choose
\eqRef{sec4.60}
w(x,t)=at+b(1+t) v(r),\;\;\forall (x,t)\in \IR^n_T,
\ee
where 
$$v(r)=\int_0^{r^2} \frac{1}{1+\tau^p}\;d\tau\quad\mbox{with}\;p=1-\frac{\st}{2}=\frac{\s-2}{2(\s-1)}.$$
Observe that $2(1-p)=\st$ and also, that $v(r)$ is like $r^2$ near $r=0$ and like $r^{\st}$ for large $r$.

In Lemma \ref{sec3.10}, we set $\bt=2$ and $\bb=\st$. Thus, parts (iv), (v) and (vi) yield
\ben
&&\mbox{(iv)}\;\;\frac{1}{\st}\le \frac{v(r)-v(1)}{r^{\st}-1}\le \frac{2}{\st},\;\forall r\ge 1,\\
&&\mbox{(v)}\;\;v^{\prime}(r)\le 2\min\left( r^{1/(\s-1)},\;r\right)\;\;\mbox{and (vi)}\;\;\frac{v^{\prime}(r)}{r}\le 2.
\een
Using the above values and expressions in (\ref{sec4.200}) or (\ref{sec4.4}) and recalling $\bar{E}$ and $\bar{F}$,  we get
\bea\label{sec4.61}
\Pc(t, w, w_t, Dw, D^2w)&\le & b(1+T)M \left(\frac{v^\prime(r)}{r}\right)+ \left[ \al ( b(1+T) )^\s\right] (v^\prime(r))^\s-a-bv(r)   \nonumber\\
&\le&2b \bar{E}+(2b)^{\s}\bar{F}\min\left( r^{\st},\;r^\s\right)-a -b v(r).
\eea
We choose
\ben
 a=2b \bar{E}+(2b)^{\s}\bar{F}+\frac{b}{\st}\quad\mbox{and}\quad 0<b< \left( \frac{1}{2^{\s}\st \bar{F}}\right)^{1/(\s-1)}.
\een
Using the above, $w$ is a super-solution in $0\le r\le 1$ and $0\le t\le T$.

In $r\ge 1$, we employ values of $a$, $b$ and the bound $v(r)\ge (r^{\st}-1)/\st$ in (\ref{sec4.61}) to find that
 \ben
 \Pc(t, w, w_t, Dw, D^2w)&\le&2b \bar{E}+(2b)^{\s}\bar{F}\min\left( r^{\st},\;r^\s\right)-a -\frac{ b \left( r^{\st}-1\right) }{\st}  \nonumber \\
 &\le&(2b)^{\s}\bar{F} r^{\st}-\frac{ b r^{\st}}{\st}
 \le b r^{\st} \left( 2^\s b^{\s-1} \bar{F}-\frac{1}{\st}\right)\le 0.
 \een
Thus, $w$ is super-solution in $\IR^n_T$. Moreover,
 \eqRef{sec4.63}
 \lim_{b\rightarrow 0} a=0.\qquad \Box
 \ee
 
 We summarize: select $w(x,t)=at+b(1+t)v(r)$ where $v(r)$ is as follows
 \bea\label{sec4.64}
&&\qquad\qquad\mbox{}\\
&&\mbox{(I) $k>1$:}\;\;v(r)=\left\{ \begin{array}{ccc} r^{\gs},& 0\le \s\le \g/2,\\ \int_0^{r^{\gs}}(1+\tau^p)^{-1}d\tau,& \s>\g/2, \end{array}\right. 
\;\;\lim_{b\rightarrow 0}a=\left\{ \begin{array}{ccc} \al,& \s=0,\\ 0,& \s>0, \end{array}\right.    \nonumber\\
&&\quad \mbox{where}\;\;p=1-(\st/\gs), \nonumber\\
&&\;\mbox{(II) $k=1$:}\;v(r)=\left\{ \begin{array}{ccc} e^{cr^2},& \s=0,\\ e^{cr}-1-cr,& 0<\s\le 1,\\ r^{\st},& 1<\s\le 2,\\
\int_0^{r^2}(1+\tau^p)^{-1}d\tau,& \s>2, \end{array}\right. \;\;\lim_{b\rightarrow 0}a=\left\{ \begin{array}{ccc} \al,& \s=0,\\ (1-\s)c^\s \bar{F},& 0<\s\le 1,\\ 
0,& \s>1. \end{array}\right.  \nonumber\\
&&\qquad \;\mbox{where}\;p=1-(\st/2). \nonumber
 \eea
See (\ref{sec4.41}), (\ref{sec4.44}) and (\ref{sec4.9}), (\ref{sec4.131}), (\ref{sec4.202}), (\ref{sec4.32}) and (\ref{sec4.63}). Recall that $v(r)$ grows like $r^{\st}$ in (I) (for $\s>\g/2$) and in (II) (for $\s>2$). 
 \vsp
\section{Sub-solutions}
\vsp
The work in this section is quite similar to that in Section 4. Although, $H$ is not assumed to be odd in $X$, the auxiliary functions used  in Section 4 continue to apply here. We will not repeat the calculations done in Section 4, instead, provide an outline as to how to use them to obtain sub-solutions. We require that the domain for $Z$ be $(-\infty, \infty).$
\vsp

We use functions of the type $w(x,t)=-\left[at+b(1+t) v(r)\right],$
where $a>0$ and $b>0$, small, $v(r)>0$ and $v^{\prime}(r)\ge 0$. Recalling (\ref{sec3.8}), we see that 
\bea\label{sec4.14}
&&\Pc(t,w, w_t, Dw, D^2w) \nonumber\\
&&\qquad\qquad=\frac{ [ b(1+t)v^{\prime}(r) ]^k}{ r } H\left( e, \left(1- \frac{r v^{\prime\prime}(r)}{v^\prime(r)} + b(1+t)rZ(w) v^\prime(r)\right)e\otimes e -I \right) \nonumber\\
&&\qquad\qquad\qquad\qquad+\chi(t) [ b (1+t) v^{\prime}(r)]^\s+a+b v(r).
\eea

We set
$$\al=\sup_{[0,T]}|\chi(t)|\;\;\;\;\mbox{and}\;\;\;N=\inf_\lambda\left( \min_{|e|=1}H(e, \lambda   e\otimes e-I)  \right).$$
We note that $N\le 0$ since $H(e, -I)\le 0$, see Condition C in Section 2.
\vsp
As done in Section 4, we take $v(r)$ to be either a power of $r$(power greater than $1$) or $e^{cr^2}$ or $e^{cr}$. For the exponential type functions, since $1-(rv^{\prime\prime}(r))/v^{\prime}(r)$ could become unbounded, a lower bound on $H$ is needed. However, $1-(rv^{\prime\prime}(r))/v^{\prime}(r)$ is bounded from below if $v(r)$ is a power of $r$ and the bound depends on the power. Since $H$ is continuous and non-decreasing in $X$, we get a natural lower bound depending on the power of $r$. We use $N$ to denote the lower bound in both situations. 

With the above discussion in mind, (\ref{sec4.14}) implies
\bea\label{sec4.140}
&&\Pc(t,w,w_t, Dw, D^2w)\ge \frac{ [ b(1+T)v^{\prime}(r) ]^k N}{ r}-\al [ b (1+T) v^{\prime}(r)]^\s+a +b v(r)  \nonumber\\
&&\quad\qquad\qquad\qquad\qquad=-\left( \frac{ [ b(1+T)v^{\prime}(r) ]^k |N|}{ r } +\al [b (1+T) v^{\prime}(r)]^\s-a-b v(r) \right).
\eea
\vsp

We now use auxiliary functions $v(r)$ that are similar to those in Section 4. The goal is to choose $a\ge 0$ and $0<b<1$ such that the expression in (\ref{sec4.140}) is non-positive i.e,
$$ \frac{ [ b(1+T)v^{\prime}(r) ]^k}{ r } |N|+\al [ b (1+T) v^{\prime}(r)]^\s-a-b v(r)\le 0.$$
The analysis is almost identical to Section 4. We list the choice for $w(x,t)$ for the various values of $\s$. 

{\bf Part I\;\;$k>1$:} Recall that $\g>2$ and $\gs=\g/(\g-2)$. Set $r=|x-z|$, for some fixed $z\in \IR^n$, and take
\ben
&&w(x,t)=\left\{ \begin{array}{ccc}  -at-b(1+t) r^{\gs}, & 0\le \s<\g/2,\\ -b(1+t) r^{\gs},& \s=\g/2,\\ -at-b(1+t)v(r), & \s>\g/2, \end{array}\right.\\
&&\mbox{where}\quad v(r)=\int_0^{r^{\gs} } \frac{1}{1+\tau^p}\; d\tau\;\; \mbox{with}\;\; \st=\frac{\s}{\s-1},\;p=1-\frac{\st}{\gs}=\frac{2\s-\g}{\g(\s-1)}.
\een 
It is easy to check that (see Remark \ref{sec3.11}) that
$$1-\frac{rv^{\prime\prime}(r)}{v^{\prime}(r)}=\frac{ 2-\gs+(2-\st) r^{p\gs} }{1+r^{p\gs}} \ge 2-\st>-\infty. $$
We choose $N$ to be an appropriate lower bound for $H$, see the right hand side of (\ref{sec4.14}). Thus, (\ref{sec4.140}) holds without any restrictions on $\inf_\lambda\left[ \min_{|e|=1}H(e, \lambda   e\otimes e-I)  \right].$
Moreover, from (\ref{sec4.64}),
\eqRef{sec4.15}
\lim_{b\rightarrow 0} a=\left\{ \begin{array}{ccc} \al, & \s=0,\\ 0,& \s>0. \end{array}\right.  \qquad \Box
\ee
\vsp
{\bf Part II\;\;$k=1$:} In this case, $\g=2$ and $k_1=0$. Set $\st=\s/(\s-1)$. We choose $a\ge 0$, $0<b<1$ and $c>0$ such that (\ref{sec4.140}) in non-positive. We select
\ben
&&w(x,t)=\left\{ \begin{array}{ccc} -a-b(1+t)e^{cr^2},& \s=0,\\ -at-b(1+t)\left( e^{cr}-1-cr\right),& 0<\s\le 1,\\ -at-b(1+t) r^{\st},& 1<\s\le 2,\\
-at-b(1+t)v(r),& 2<\s<\infty, \end{array}\right.\\
&&\mbox{where}\;\;v(r)=\int_0^{r^2} \frac{1}{1+\tau^p}\;d\tau\;\;\mbox{with}\;\; p=1-\frac{\st}{2}=\frac{\s-2}{2(\s-1)}.
\een
\vsp
If $0\le \s\le 1$ then $1-rv^{\prime\prime}(r)/v^{\prime}(r)\le 0$ becomes unbounded as $r\rightarrow \infty$. Thus, we impose that $|\inf_\lambda\left[ \min_{|e|=1}H(e, \lambda   e\otimes e-I)  \right] |<\infty.$ For $\s>1$, however, no such requirement is made. 

Moreover, from (\ref{sec4.64}), 
\eqRef{sec4.16}
\lim_{b\rightarrow 0} a=\left\{ \begin{array}{ccc} \al, & \s=0,\\ (1-\s) \al (c(1+T))^\s,& 0<\s\le 1,\\ 0, & \s>1. \end{array}\right.\qquad \Box
\ee
\vspp
\section{Some Special cases}
\vsp

In this section we consider some special cases. Recall that
\eqRef{sec6.0}
\Pc(t,w, w_t, Dw, D^2w)=H(Dw, D^2w+Z(w)Dw\otimes Dw)+\chi(t) |Dw|^\s-w_t.
\ee
As before, set
$$N=\inf_\lam \left[  \min_{|e|=1}H(e,\lam e\otimes e-I) \right].$$ 
\vsp
We discuss the following three cases.

{\bf Case (i):} $k\ge 1$ and $\chi\equiv 0$. The equations reads
$$H(Dv, D^2v+Z(v)Dv\otimes Dv)-v_t=0,\;\;\mbox{in $\IR^n_T$, $v>0$, with $v(x,0)=h(x),\;\forall x\in \IR^n$.}$$

As observed in (\ref{sec1.1}), (\ref{sec1.2}) and part (f) of Remark \ref{sec3.19}, this applies to the doubly nonlinear case by employing a change of variables. Moreover, as noted in
Remark \ref{sec3.19} and Lemma \ref{sec3.20}, the convergence or the divergence of the integral 
$$I=\int_0^1 f^{-1/(k-1)}(\tht) \; d\tht,\;\;k>1, $$
determines the domain of $Z$. In particular, if $I<\infty$ then the domain of $Z$ is $(0,\infty)$ or $[0,\infty)$, and if $I=\infty$ then the domain is $(-\infty, \infty)$. 

The super-solutions in Section 4, (in particular, the one in Sub-Part (i) or Part I) being positive, are also super-solutions of (\ref{sec6.0}) regardless the domain of $Z$. However, the domain of $Z$ needs to be stated more precisely for sub-solutions. If the integral $I$ diverges then the work in Section 5, in particular, Part I applies since the domain of $Z$ is $(-\infty, \infty)$. If $I$ converges then the domain is $(0,\infty)$ or $[0,\infty)$ and a different sub-solution needs to be calculated. We do this in this section. 

We also include here the case $k=1$ where $Z$ is defined on $(0,\infty)$ or $[0,\infty)$. The two Part II's in Sections 4 and 5 address the case where the domain is $(-\infty, \infty)$.
\vsp
The next two cases bring out the influence of the sign of $\chi$. 

{\bf Case (ii):} We discuss super-solutions in the case $\chi\le 0$ and we derive a maximum principle. 
 
{\bf Case (iii):} We study sub-solutions for
$\chi\ge 0$ and this leads to a minimum principle.

The cases {\bf (ii)} and {\bf (iii)} are related.
\vsp

Let $z\in \IR^n$ be a fixed and set $r=|x-z|,\;\forall x\in \IR^n$.

We begin with Case (i).
\vsp
{\bf Case (i-1):} We take $k>1$, $\chi\equiv 0$, $\s=0$ and assume that the domain of $Z$ contains $(0,\infty)$. Thus, the equation reads
$$\Pz(t,w, w_t, Dw, D^2w)=H(Dw, D^2w+Z(w)Dw\otimes Dw)-w_t.$$

Since our goal is to construct positive sub-solutions $w$, it suffices to find a $w$ such that
$H(Dw, D^2w)-w_t\ge 0$ since ellipticity ($Z\ge 0$) implies the desired conclusion.  

Let $R>0$ and set $B^R_T=B_R(z)\times (0,T)$. We construct a sub-solution $w$ for any large $R$. More precisely, $w\ge 0$ solves 
$$H(Dw, D^2w)-w_t\ge 0,\;\;\mbox{in $B^R_T$ and $w(x,0)\le g(x),\;\forall x\in B_R(z)$.}$$

We define
\eqRef{sec6.900}
w(x,t)=\psi(t)v(r)= \frac{D \left[ R^{ (k+1)/k}-r^{(k+1)/k} \right]^{k/(k-1)} } { (E+t)^{1/(k-1)} },\;\;\forall(x,t)\in B^R_T,
\ee
where $D,\; E>0$ are to be determined.
One recalls from (\ref{sec3.7}) that if $w=\psi(t) v(r)$, with $w_r\le 0$, then
\bea\label{sec6.90}
&&H(Dw, D^2w)-w_t=\frac{  \left( |\psi(t) v^{\prime}(r)| \right)^k }{r} H\left(e, \left( 1-\frac{rv^{\prime\prime}(r)}{v^{\prime}(r)  } \right)e\otimes e-I \right)-v(r)\psi^{\prime}(t) \nonumber\\
&&\qquad\qquad\qquad\qquad\ge -\frac{|N|  \left( |\psi(t) v^{\prime}(r)| \right)^k }{r}-v(r)\psi^{\prime}(t).
\eea
Using the expression for $w$ and setting $c_k= \left[ (k+1)/(k-1) \right]^k$, we see that
\ben
&&-v(r)\psi^{\prime}(t)- \frac{ |N| \left( |\psi(t) v^{\prime}(r)| \right)^k }{r}=\frac{ D v(r)}{ (k-1) ( E+t)^{k/(k-1)} }-\frac{  c_k |N| D^k v(r)}{ (E+t)^{k/(k-1)} }  \\
&&\qquad\qquad\qquad\qquad\qquad\quad=\frac{ D v(r)}{ (k-1) ( E+t)^{k/(k-1)} } \left[ 1- (k-1) c_k |N| D^{k-1} \right].
\een
Choosing
$$D=  \left(\frac{ 1}{c_k(k-1) |N|  } \right)^{1/(k-1)},$$
and using the above in (\ref{sec6.90}), we get a sub-solution $w\ge 0$ in $B^T_R$ such that $w(R,t)=0$. Next, we calculate $E$ by requiring that  
$$w(z,0)=w(0,0)=\frac{ D R^{(k+1)/(k-1)} }{ E^{1/(k-1)} }=\inf_x h(x)=\mu.$$
Thus,
\ben
w=\frac{ D R^{(k+1)/(k-1)}} {E^{1/(k-1)} } \frac{  \left[ 1-(r/R)^{(k+1)/k} \right]^{k/(k-1)} } { (1+(t/E) )^{1/(k-1)} } =  \frac{ \mu \left[ 1-(r/R)^{(k+1)/k} \right]^{k/(k-1)} } { (1+(t/E) )^{1/(k-1)} }.
\een
Note that $E=O(R^{k+1})$ and 
$$
w(z,t)=w(0,t)= \frac{\mu }{   (1+(t/E) )^{1/(k-1) } }\rightarrow \mu \;\;\mbox{as $R\rightarrow \infty$}.
$$
We record that in $0\le r<R$, 
\eqRef{sec6.91}
w(x,t)= \frac{ \mu \left[ 1-(r/R)^{(k+1)/k} \right]^{k/(k-1)} } { (1+(t/E) )^{1/(k-1)} },\;\;\mbox{where}\;\;E=\frac{R^{k+1}}{ c_k  \mu^{k-1}(k-1)|N|}.
\ee
\vsp
{\bf Case (i-2):} We now study $k=1$. We take $w(x,t)=D e^{-Er^2}e^{-Ft}$ and recall (\ref{sec6.90}). We get
\ben
&&-|N| \frac{  \psi(t)| v^{\prime}(r)| }{r}-v(r)\psi^{\prime}(t)\\
&&\qquad\qquad\qquad\qquad=DF e^{-Er^2}e^{-Ft}-|N| 2DE e^{-Er^2} e^{-Ft}=De^{-Er^2} e^{-Ft} \left( F-2|N| E \right)
\een
We take $F=2|N|E$ and $D=\mu$ and obtain a sub-solution
\eqRef{sec6.92}
w(x,t)=\mu e^{-Er^2} e^{-2|N| E t},\;\;\forall E>0.\ee

It is clear that $W\rightarrow \mu$ as $E\rightarrow 0.$  $\Box$
\vspp
{\bf Case (ii):} We consider 
\eqRef{sec6.1}
\Pc(t, w, w_t, Dw, D^2w)=H(Dw, D^2w+Z(w) Dw\otimes Dw)+\chi(t)|Dw|^\s-w_t
\ee
where $\chi\le 0$. We set 
$$\hat{\al}=\sup_{(0,T)}\chi(t)$$
and assume that $\hat{\al}<0$. We further assume that 
\eqRef{sec6.2}
k\ge 1\quad \mbox{and}\quad \s\ge k.
\ee
Our goal here is to construct super-solutions $w\ge 0$, i.e., $\Pc(t, w, w_t, Dw, D^2w)\le 0$ in cylinders $B^R_T$. 

Selecting $w(x,t)=at+(1+t)v(r),\;v^\prime\ge 0$,  setting $M=\sup_\lam \left[ \max_{|e|=1} H(e, I+\lam e\otimes e) \right]$ and recalling (\ref{sec3.6}) and (\ref{sec6.1}) we find that
\bea\label{sec6.3}
&&\Pc(t,w,w_t,Dw,D^2w)\nonumber\\ 
&&\qquad\le \frac{ [(1+t) v^\prime(r)]^k}{r} H\left( e, \;I +\left( \frac{r v^{\prime\prime}(r)}{v^\prime(r)}-1 + (1+t) r v^{\prime}(r) Z(w) \; \right)e\otimes e\;\right) \nonumber\\
&&\quad\qquad\qquad\qquad-|\hat{\al}| [(1+t) v^\prime(r)]^\s-a-v(r) \nonumber\\
&&\qquad\le \frac{  [(1+t) v^\prime(r)]^k }{r} \left\{  M -|\hat{\al}| [(1+t) v^\prime(r)]^{\s-k} r \right\}-a- v(r).
\eea

For $R>0$, set 
$$v(r)=\left( R^2-r^2\right)^{-1},\;\;0\le r<R.$$
Since 
$$v^{\prime}(r)=(2 r) ( R^2-r^2 )^{-2},$$ 
(\ref{sec6.3}) yields that, in $0\le r<R$, 
\bea\label{sec6.4}
&&\Pc(t,w,w_t,Dw,D^2w) \le \frac{ [2(1+t)]^k r^{k-1}  }{ (R^2-r^2)^{2k}} \left( M-   |\hat{\al}|   \left( \frac{ 2 (1+t) }{ (R^2-r^2)^2 } \right)^{\s-k} r^{\s-k+1}  \right)-a  \nonumber \\
&&\qquad\qquad\qquad\qquad\qquad\qquad-\frac{1}{ R^2-r^2 }.
\eea
\vsp
{\bf Sub-Case (ii-1) ($\s=k$:)} Set $r^*=M/\hat{\al}$ and take $R>r^*$. Then (\ref{sec6.4}) yields that
\eqRef{sec6.5}
\Pc(t,w,w_t,Dw,D^2w) \le \left( \frac{ 2(1+t)}{ (R^2-r^2)^{2} }\right)^k \left(  M  -  |\hat{\al}| r \right)r^{k-1}-a-\frac{1}{ R^2-r^2}.
\ee
Select
$$a=M\left( \frac{ 2(1+T) }{ (R^2-(r^*)^2)^{2} }\right)^k  (r^*)^{k-1}.$$
With this choice (\ref{sec6.5}) shows that $w$ is a super-solution in $B^R_T$. 
Thus, 
\eqRef{sec6.6}
w(x,t)=at+(1+t) v(r)\;\;\mbox{and}\;\;\lim_{R\rightarrow \infty}a=0.\qquad \Box
\ee
\vsp
{\bf Case (ii-2) ($\s>k$:)} From (\ref{sec6.4}) we have that
\ben
&&\Pc(t,w,w_t,Dw,D^2w) \nonumber\\
&&\qquad \qquad \le \left(  \frac{ 2(1+t) }{ (R^2-r^2)^{2} }\right)^k  \left( M - |\hat{\al}|  \left( \frac{2 }{ (R^2-r^2)^{ 2 } } \right)^{\s-k} r^{\s-k+1} \right)r^{k-1}-a.
\een
Since the function $f(r)=r^{\s-k+1} \left( R^2-r^2 \right)^{2(k-\s)}$ is  continuous and  increasing in $0\le r<R,$  f(0)=0 and $f(r)\rightarrow \infty$, as $r\rightarrow R$, there is an $r^*=r^*(R)<R$ such that 
$2^{\s-k}  |\hat{\al}|  f(r^*)=M$. Choose 
$$a=M\left(  \frac{ 2 (1+T) }{ (R^2-(r^*)^2)^{2} }\right)^k (r^*)^{k-1}.$$
Clearly, $w$ is super-solution in $0\le r<R$.

Next, we recall that
$$f(r^*)=  \frac{ (r^*)^{\s-k+1}  }{ [R^2-(r^*)^2]^{2(\s-k)} }   =\frac{ M}{ 2^{\s-k} |\hat{\al}| } \; .$$
Clearly, $r^*\rightarrow \infty$, as $R\rightarrow \infty$. For calculating $\lim_{R\rightarrow \infty} a$, we use the formula for $f(r^*)$ and observe that for an appropriate constant $D$, we have
$$ \frac{ (r^*)^{k-1} } { [ R^2-(r^*)^2]^{2k} }=\frac{ D (r^*)^{k-1}}{ (r^*)^ { k(\s-k+1)/(\s-k) } }=\frac{D}{ (r^*)^{1+k/(\s-k)}}.$$
Thus,
\eqRef{sec6.7}
\lim_{R\rightarrow \infty}a=0.\qquad \Box
\ee

\vspp
{\bf Case (iii) Sub-solution:} 
We construct a function $w(x,t)$ such that 
\ben
&&\Pc(t, w, w_t, Dw, D^2w)\\
&&\qquad\qquad=H(Dw, D^2w+Z(w)Dw\otimes Dw)+\chi(t)|Dw|^\s-w_t\ge 0,\;\;\mbox{in $\IR^n_T$,}
\een 
where $\chi\ge 0$. Set
$$N=\inf_\lam \left[ \min_{|e|=1} H(e, \lam e\otimes e-I) \right]\qquad\mbox{and}\qquad \hat{\al}=\inf \chi(t).$$

Select $w(x,t)=-at-(1+t)v(r),\;v^{\prime}\ge 0,$ and recall (\ref{sec3.8}):  
\bea\label{sec6.8}
&&\Pc(t,w, w_t, Dw, D^2w) \nonumber\\
&&\qquad\qquad=\frac{ [ (1+t)v^\prime(r)]^k}{ r } H\left( e, \;\left(\;  (1+t)rv^{\prime}(r)Z(w)\;+1   - \frac{rv^{\prime\prime}(r)}{v^\prime(r)}\right)e\otimes e -I \right) \nonumber\\
&&\quad\qquad\qquad\qquad +\chi(t) [  (1+t) v^{\prime}(r) ]^\s+a+  v(r) \nonumber\\
&&\qquad\qquad \ge - \left[ \frac{ [ (1+t)v^\prime(r)]^k}{ r } |N| -\chi(t) [  (1+t) v^{\prime}(r) ]^\s-a-  v(r) \right].
\eea
\vsp
Defining
$$v(r)=\frac{1}{ R^2-r^2 },\;\;\forall\,\, 0\le r<R,$$
and proceeding as in Case (ii), one can construct a sub-solution $w$ with the same properties. $\Box$

\begin{rem}\label{sec6.100} We point out that, except for Case (i) in this section all the auxiliary functions in this work are of the kind $w(x,t)=at+b(1+t)v(r)$, where $v(r)$ is an appropriately chosen function, $r=|x-z|,\;\forall x\in \IR^n$ and $z\in \IR^n$ is fixed. 

Case (i) is used in proving the minimum principle in Theorem. For $k>1$ we utilize $w$ in (\ref{sec6.91}) and for $k=1$ we use $w$ in (\ref{sec6.92}). Note that $k>1$ requires no lower bound except $u>0$, however, for $k=1$ we assume a lower bound. 

Case (ii) implies a maximum principle without any imposition of an upper bound. Case (iii) leads to a minimum principle without requiring any lower bound.

We provide details in Section 7. $\Box$
\end{rem}
\vsp
\section{Proofs of the main results} \label{Section 7}
\vsp
Assume that $-\infty<\inf_{\IR^n} g\le \sup_{\IR^n} g<\infty$ and set
$$\mu=\inf_{\IR^n}g\;\;\;\mbox{and}\;\;\;\nu=\sup_{\IR^n}g.$$
\vsp

\NI{\bf Proofs of Theorems \ref{sec2.6} and \ref{sec2.7}: ($k>1$)}
\vsp
We first present the proof of Theorem \ref{sec2.6}. 
 Select $\ve>0$ small and $R_0>0$ such that
\eqRef{sec7.1}
\sup_{ [0,R]\times[0,T]}u(x,t)\le \ve R^{\bt},\;\;\;\forall\;R\ge R_0.
\ee
where $\bt$ is as described in the statements of the theorem.
  
Recall from (\ref{sec4.4.0}) and (\ref{sec4.64})  that the super-solution $w(x,t)$ can be written as
\ben
w(x,t)=at+b(1+t){v}(r),
\een
for an appropriate ${v}(r)>0$. Observe that $w$ is a super-solution for any small $b>0$. Also, ${v}$ grows like $r^\bt$, see (\ref{sec4.64}) and the constructions of the super-solutions in Section 4. Define
$$W(x,t)=\nu+w(x,t).$$

Let $\hat{k}>2$ be a constant so that $\hat{k}v\ge r^\bt$ for $r\ge R_1$, where $R_1$ is large.  We take $b=\hat{k}\ve$ in $W(x,t)$ and consider the cylinder $B_{R}(z)\times [0,T]$, 
where $R\ge$ max$(R_0, R_1)$. At $t=0$, $W(x,0)=\nu+\hat{k}\ve {v}(r)\ge  \nu \ge u(x,0)$. On $|x-z|=R$,
$$W(x,t)\ge \hat{k}\ve {v}(R)\ge \ve R^\bt.$$
Thus, $W\ge u$ on the parabolic boundary of $B_{R}(z)\times (0,T).$ We apply Lemma \ref{sec3.18} to conclude that $W\ge u$ in $B_R(z)\times(0,T)$ for any $R$, i.e.,
\ben
u(x,t)\le at+\hat{k}\ve(1+t)v(r)+\nu,\;\;\forall |x-z|\le R.
\een
Taking $x=z$, we get that $u(z,t)\le at+\nu$. Letting $R\rightarrow \infty$ and then $\ve\rightarrow 0$ (i.e. $b\rightarrow 0$) and using (\ref{sec4.64}) (employ $\lim_{b\rightarrow 0} a$) we obtain the conclusion of the Theorem. 

The proof of Theorem \ref{sec2.7} can be obtained by using Part I of Section 5 and arguing analogously. We omit the details.  
\hfill $\Box$
\vsp
\NI{\bf Proofs of Theorems \ref{sec2.9} and \ref{sec2.10}: ($k=1$)} 

We first prove Theorem \ref{sec2.9}. We recall (II) in (\ref{sec4.64}). 

We take $\s=0$. Let $0<\ve<c/10$ be small  and fixed. Set
$$W(x,t)={\nu}+\al t+{ \ve}(1+t)e^{cr^2},\;\forall (x,t)\in \IR^n_T.$$
Then $W$ is super-solution for any small {$\ve>0$. }

Choose $R_0>0$ such that $\sup_{B_R(z)\times [0,T]} u(x,t)\le e^{\ve R^2}$ and
$\ve e^{cR^2}>e^{\ve R^2},\;\forall R>R_0$.

We apply the comparison principle Lemma \ref{sec3.18} to prove the claim in the theorem. 
Observe that $W(x,0)\ge { \nu}\ge u(x,0),\;\forall x\in \IR^n_T.$ On $|x-z|=R>R_1$, $W(x,t)\ge \ve e^{cR^2}\ge e^{\ve R^2}.$ By Lemma \ref{sec3.18}, $u(x,t)\le W(x,t),\;\forall (x,t)\in B_R(z)\times (0,T),$ for any $R>R_0$. Hence,
$$u(z,t)\le W(z,t)={ \nu +}\al t+\ve(1+t){ e^{\ve r^2}}.$$
Since the above holds for any large $R$, we let $\ve\rightarrow 0$ to obtain the claim in part (a).

Part (b) may now be shown by arguing as above. Part (c) may be shown by following the ideas in the Proof of Theorem \ref{sec2.6}. Theorem \ref{sec2.10} follows analogously, see Part II in Section 5.  $\Box$

\vsp
We now present the proof of Theorem \ref{sec2.8}. We start with the maximum principle. 
\vsp
{\bf Proof of Theorem \ref{sec2.8}(a): (Maximum principle)} We refer to Remark \ref{sec3.19} and the comparison principle in Lemma \ref{sec3.20}. 
We set $\al=0$ in part (a) of Theorem \ref{sec2.6}. Suppose that $u=\phi(v)$ where the change of variable is as in Remark \ref{sec3.19}. If 
$$H(Du, D^2u)-f(u)u_t\ge 0,\;\;\mbox{in $\IR^n_T$, $u>0$, with $u(x,0)\le g(x),\;\forall x\in \IR^n$,}$$
then
$$H(Dv, D^2v+Z(v) DV\otimes Dv)-v_t\ge 0,\;\;\mbox{in $\IR^n_T$, with $v(x,0)\le \phi^{-1}(g(x)),\;\forall x\in \IR^n$,}$$
where $Z(s)=\phi^{\prime\prime}(s)/\phi^{\prime}(s)$ and the domain of $Z$ contains $(0,\infty)$. See Remark \ref{sec3.19}. 

The super-solution $w$ used in the proof of Theorem \ref{sec2.6} is positive. Clearly, $Z(w)$ is well-defined. Using Lemma \ref{sec3.18} (or Lemma \ref{sec3.20}) and  
arguing as in the proof of Theorem
\ref{sec2.6} we get that $v\le \sup_x \phi^{-1}(g(x))$, if 
$$ \sup_{B_R(z)\times (0,T)} v(x,t)=o( R^{\gs})\;\;\mbox{as $R\rightarrow \infty$.}$$
Thus, the claim holds for $u$.
\vsp
{\bf Proof of Theorem \ref{sec2.8}: (Minimum principle)} 

(i) Suppose that 
$$\lim_{\dl\rightarrow 0^+}F(1)-F(\dl)<\infty.$$
We choose
$$v=\phi^{-1}(u)=\int_0^u  f^{-1/(k-1)}(\tht)d\tht,\;\;u>0.$$
Then
$$H(Dv, D^2v+Z(v) Dv\otimes Dv)-v_t\le 0,\;v>0,\;\mbox{in $\IR^n_T$ with $v(x,0)\ge \phi^{-1}(g(x)),\;\forall x\in \IR^n,$}$$
where the domain of $Z$ contains $(0, \infty)$. We recall Case(i-1) from Section 6 and (\ref{sec6.91}) i.e.,
$$w(x,t)= \frac{ \hat{\mu} \left[ 1-(r/R)^{(k+1)/k} \right]^{k/(k-1)} } { (1+(t/E) )^{1/(k-1)} },\;\;\mbox{where}\;\;E=\frac{R^{k+1}}{ c_k  {\hat\mu}^{k-1}(k-1)|N|},$$
for any large $R>0$. Here $\hat{\mu}=\phi^{-1}(\mu).$

We use comparison in $B_R(z)\times [0,T]$. It is clear that $v(x,0)\ge \phi^{-1}(g(x)){ \ge} w(x,0),$ $\forall |x-z|<R.$ Since $v>0$ in $\IR^n_T$, working with $R^{'}<R$, close to $R$, we see that
$v(x,t)\ge w(x,t)$. Applying Lemma \ref{sec3.18} to the parabolic boundary of $B_{R^{'}}(z)\times (0,T)$, we get that $v(x,t)\ge w(x,t)$ in $B_{R^{'}}(z)\times (0,T)$. Thus,
$$v(z,t)\ge w(z,t)=\frac{ \hat{\mu} } { (1+(t/E) )^{1/(k-1)} },$$
Letting $R\rightarrow \infty$ (i.e. $E\rightarrow \infty$), we get that $v(z,t)\ge \hat{\mu}$ and the claim follows for $u$.
For $k=1$ and $f\equiv 1$, we use Case (i-2) in Section {{6}} and (\ref{sec6.92}) and assume that $\inf_{B_R(z)\times [0,T]} u(x,t)\ge  { \mu} e^{-\ve R^2}$, where $R>0$ is large enough and $\ve>0$ is small but fixed. Recall from (\ref{sec6.92}) that
$$w(x,t)=\mu e^{-Er^2} e^{-2|N| E t},\;\;\forall E>0,$$
is a sub-solution in $\IR^n_T$. Working in cylinders $B_R(z)\times (0,T)$, for large $R$, we find that $u(x,0)\ge \mu\ge w(x,0)$, for any $E>0$. Fix an $E>\ve$. On $|x-z|=R$, 
$w(x,t)\le u(x,t)$ implying that $w(x,t)\le u(x,t)$ in $B_R\times (0,T)$, for any large $R$, and, hence, in $\IR^n_T$. Thus,
$$w(z,t)=\mu e^{-2|N|E t}\le u(z,t),\;\forall E>\ve.$$
Since the above holds for any $R$ and, hence, for any $\ve>0$, we get that the above estimate holds for any $E>0$. Clearly, the claim holds.
\vsp
(ii) Suppose that 
$$\int_0 ^1 f^{-1/(k-1)}(\tht)d\tht<\infty.$$
We choose
$$v=\phi^{-1}(u)=\int_0^u  f^{-1/(k-1)}(\tht)d\tht,\;\;u>0.$$
Then
$$H(Dv, D^2v+Z(v) Dv\otimes Dv)-v_t\le 0,\;\;\mbox{in $\IR^n_T$ with $v(x,0)\ge \phi^{-1}(g(x)),\;\forall x\in \IR^n,$}$$
where the domain of $Z$ contains $(0, \infty)$. This is similar to the proof of Theorem \ref{sec2.10}. 

The case $k=1$ and $f\equiv 1$ also follows in an analogous way. $\Box$

\vsp
\NI Department of Mathematics, Western Kentucky University, Bowling Green, Ky 42101, USA\\
\NI Department of Liberal Arts, Savannah College of Arts and Design, Savannah, GA 31405, USA


\begin{thebibliography}{99}

\bibitem{AJK} G. Akagi, P. Juutinen and R. Kajikiya, \it Asymptotic behavior of viscosity solutions for a degenerate parabolic equation associated with the infinity-Laplacian, 
\rm Math.Ann. 343 (2009), no 4, 921-953.

\bibitem{BM1} T. Bhattacharya and L. Marazzi, {\it On the viscosity solution to a parabolic equation,} Annali di Matematica Pura ed Applicata, vol 194, no 5, 2014.
DOI:10.1007/s10231-014-0427-1

\bibitem{BM2} T. Bhattacharya and L. Marazzi, {\em On the viscosity solutions to Trudinger's equation},
Nonlinear Differential equations and applications (NoDEA), vol 22, no 5, 2015.
DOI:10.1007/s00030-015-0315-4

\bibitem{BM3} T. Bhattacharya and L. Marazzi, {\em Asymptotics of viscosity solutions of some doubly nonlinear parabolic eqns.} J. of Evol. Eqns. {\bf 16} (4), 759--788. (2016) DOI:10.1007/s00026-015-0319-x

\bibitem{BM4} T. Bhattacharya and L. Marazzi, {{\it Errata to "On the viscosity solutions to Trudinger's equation".} Nonlinear Diff. Eqns. Appl.(2016) 23:68}

\bibitem{BM5} T. Bhattacharya and L. Marazzi, {\it On the viscosity solution to a class of nonlinear degenerate Parabolic differential equations.} {Rev. Mat. Complut. (2017) 30:621--656.}

\bibitem{BM6} T. Bhattacharya and L. Marazzi, {\it A Phragm\'en-Lindel\"of property of viscosity solutions to a class of doubly nonlinear parabolic equations I.} Preprint arxiv.org/abs/1805.05861 (2018). 

\bibitem{CIL} M. G. Crandall, H. Ishii and P. L. Lions,\it User's guide to viscosity solutions of second order partial differential equations,\rm Bull. Amer. Math. Soc. 27(1992) 1-67.

\bibitem{ED} E. DiBenedetto, \it Degenerate Parabolic Equations, \rm Universitext, Springer (1993) 

\bibitem{JL}  P. Juutinen and P. Lindqvist, \it Pointwise decay for the solutions of degenerate and singular parabolic equations, \rm Adv. Differential Equations 14(2009), no. 7-8, 663-684.

\bibitem{TR} N. S. Trudinger, \it Pointwise estimates and quasilinear parabolic equations, \rm Comm. Pure Appl. Math. 21, 205-226 (1968) 

\bibitem{Tych} A. Tychonoff, {\em Th\'eor\`emes d'unicit \'e pour l’ \'equation de la chaleur,}
Mat. Sb., 1935, Volume 42, Number 2, 199-216.

\end{thebibliography}
\end{document}